\documentclass[11pt,reqno]{amsart}

\usepackage{graphics,wrapfig, graphicx, mathrsfs,xcolor}
\usepackage{hyperref}
\usepackage{amsfonts,amsthm,amsmath,amssymb,amscd,mathrsfs,graphicx,xcolor,subfig,tikz}
\usepackage{graphics}
\usepackage{indentfirst,centernot}
\usepackage{cite}
\usepackage{bm, enumerate,xcolor}
\usepackage[dvips]{epsfig}
\usepackage{latexsym}
\usepackage{float}
\usepackage{mathtools}
\usepackage{amssymb,amscd,graphicx,xcolor,subfig,tikz}
\usepackage{graphics}
\usepackage{indentfirst}
\usepackage{bm, enumerate,xcolor}
\usepackage[dvips]{epsfig}
\usepackage{latexsym}
\usepackage{float}
\usepackage{amsmath}
\usepackage{amssymb}

\newtheorem{theorem}{Theorem}[section]

\newtheorem{remark}{Remark}[section]

\newcommand{\dd}{\mathrm{d}}
\newcommand{\T}{\mathbb{T}}
\numberwithin{equation}{section} \numberwithin{equation}{section}
\numberwithin{example}{section} \numberwithin{remark}{section}
\numberwithin{figure}{section} \numberwithin{algorithm}{section}

\def\ba{\begin{array}}
\def\ea{\end{array}}
\def\bma{\left(\begin{matrix}}
\def\ema{\end{matrix}\right)}
\def\be{\begin{equation}}
\def\ee{\end{equation}}

\theoremstyle{definition}
\newtheorem*{key-original-step}{Key Original Step}
\newenvironment{heuristics}
  {\par\noindent\textbf{Heuristics:}\begin{quote}\itshape}
  {\end{quote}}

\begin{document}
\title[Lower Bounds by AI]{Lower Bounds for Advection-Diffusion Equations: An Exploration with AI-Generated Proofs}

\author[C. An]{Chenyang An}
\address{Chenyang An: University of California, San Diego, La Jolla, United States}
\email{cya.portfolio@gmail.com}

\author[X. Xu]{Xiaoqian Xu}
\address{Xiaoqian Xu: Duke Kunshan University, Zu Chongzhi Center, No. 8 Duke Avenue, Kunshan, Jiangsu Province, 215316, P.R. China}
\email{xiaoqian.xu@dukekunshan.edu.cn}

\date{\today}

\subjclass[2020]{}%

\keywords{}%

\thanks{}
\begin{abstract}
We establish explicit lower bounds for advection-diffusion equations in three settings: a polynomial $\dot H^{-1}$ bound for inviscid shears with $u\in L^\infty_t W^{1,1}_y$, a uniform positive lower bound on the mixing scale for diffusive shears, and an exponential $L^2$ bound for rapidly oscillating time-periodic flows. All constants are explicit in the data. 

The proofs were generated entirely by a multi-agent math proving system, QED \cite{QED}, without expert human intervention, serving as a test of AI's capability to produce rigorous mathematics.

\end{abstract}
\maketitle

\section{Introduction}
\label{sec:intro}

The advection--diffusion equation
\begin{equation}\label{eq:ad}
\partial_t \rho + \mathbf{u} \cdot \nabla \rho = \mu \Delta \rho
\end{equation}
on the torus $\mathbb{T}^d$ provides a fundamental setting in which to study the competition between transport and dissipation. When $\mathbf{u}$ is divergence-free, the dynamics are driven by two opposing effects: advection stretches and folds the scalar field, creating ever finer scales, while diffusion smooths the field and damps gradients. Understanding the long-time balance between these mechanisms is at the heart of the mathematical theory of mixing, with connections to turbulence theory~\cite{batchelor1959small}, chemical reactor design~\cite{danckwerts1952definition}, and atmospheric science~\cite{pierrehumbert1994tracer}.

For a mean-zero solution the $L^2$ norm satisfies
\[
\frac{\dd}{\dd t}\|\rho(t)\|_{L^2}^2 = -2\mu\|\nabla \rho(t)\|_{L^2}^2,
\]
so the norm is non-increasing.  In the absence of advection the decay is exponential, governed by the Poincaré inequality.  The presence of a velocity field can accelerate this dissipation, a phenomenon termed enhanced dissipation, which has been the subject of extensive research~\cite{constantin2008diffusion,bedrossian2017enhanced,wei2021diffusion,wei2020linear,kiselev2016suppression,feng2019dissipation,zelati2020relation,thiffeault2012using}.

The question of optimal decay rates for the $L^2$ norm has attracted considerable attention. A fundamental lower bound was proved in \cite{poon1996qnique, hongjie2017, miles2018diffusion}, that is for any
bounded velocity field,
\[
\|\rho(t)\|_{L^2} \ge C \exp(-C e^{Ct}),
\]
so the norm cannot drop faster than a double exponential. Rowan~\cite{rowan2025superexponential} subsequently proved that this bound is sharp by constructing bounded flows that indeed achieve double-exponential decay. Thus, without additional structure, the $L^2$ norm may decay doubly-exponentially, but not faster.

A natural question is what structural conditions restore the more familiar exponential lower bound. For shear flows, Huang and Xu~\cite{huang2025exponential} recently proved that an exponential $L^2$ lower bound is possible, confirming that diffusion can fundamentally suppress mixing. Their analysis, however, is non-constructive, leaving open the explicit dependence of the decay exponent on the data. In this paper we develop several new directions related to these lower bounds, with a unified focus on explicit, constructive estimates. We treat three increasingly general settings, each exposing a different aspect of the interplay between advection, diffusion, and temporal oscillation.

In this paper, first, we study the purely advective limit of a shear flow, where the absence of diffusion leads to a polynomial lower bound on the $\dot H^{-1}$ norm.  Second, we add a small diffusivity and consider a bounded shear; here we obtain an explicit exponential $L^2$ lower bound and a uniform positive lower bound on the mixing scale $\|\rho\|_{\dot H^{-1}}/\|\rho\|_{L^2}$, which is shown to be sharp in the $\nu$-scaling.  This gives a rigorous foundation for the arrested cascade observed numerically in~\cite{miles2018diffusion} and connects to the broader theory of mixing scales~\cite{mathew2005multiscale,thiffeault2012using}. Finally, we consider general divergence-free velocity fields with fast time-periodic oscillations, proving that for sufficiently large frequencies an exponential $L^2$ lower bound is restored, with explicit constants.  As a by-product, this also provides an exponential bound for time-independent flows.

\medskip
\noindent\textbf{Inviscid shear mixing.}
Consider the transport equation
\[
\partial_t \theta + u(y,t)\,\partial_x\theta = 0,\qquad \theta(0)=\theta_0\in C^\infty(\T^2),\ \int \theta_0=0,
\]
where the shear satisfies $u\in L^\infty_t W^{1,1}_y(\T)$. Lower bounds on negative Sobolev norms for inviscid shears have been studied under several different regularity assumptions. For $C^\alpha$ shear flows, Colombo, Coti Zelati, and Widmayer~\cite{colombo2020mixing} proved that the $L^2_x H^{-1}_y$ norm satisfies a uniform upper bound of order $t^{-1}$ for all times, together with a matching lower bound along a sequence of times $t'_m\to\infty$:
\[
\|\theta_\star(t'_m)\|_{L^2_x H^{-1}_y} \ge C (t'_m)^{-1}\|\theta^{\mathrm{in}}_\star\|_{L^2_x H^1_y}.
\]
Galeati and Gubinelli~\cite{galeati2023mixing} considered generically rough shears in the Besov scale $B^\alpha_{1,\infty}$ and established sharp mixing rates in the $H^{1/2}$ norm, uniform for all times; along certain diverging sequences $\tilde t_n$, they also proved a lower bound for the stronger $H^{-1}$ norm:
\[
\|e^{-\tilde t_n u\partial_x}f_0\|_{H^{-1}} \gtrsim \tilde t_n^{-1}\|f_0\|_{H^1}.
\]
For cellular flows with a palenstrophy constraint $u\in L^\infty_t \dot W^{s,p}$ ($s>1$), Crippa and Schulze~\cite{crippa2017cellular} obtained a polynomial lower bound on the $\dot H^{-1}$ norm along a discrete sequence of times:
\[
\|\theta(T_n)\|_{\dot H^{-1}} \ge C T_n^{-2/(s-1)}.
\]

Our contribution is to extend this line of inquiry to the still weaker $W^{1,1}$ regularity class, and to obtain a bound that holds for \emph{all} times rather than along a subsequence.

\begin{theorem}[Polynomial $\dot H^{-1}$ lower bound for inviscid shears]\label{thm:inviscid}
Let $\theta$ be the unique solution of the above transport problem with $\theta_0\not\equiv0$. Then there exists a constant $c_*>0$, depending only on $\theta_0$ and $u$, such that
\[
\|\theta(t)\|_{\dot H^{-1}} \ge \frac{c_*}{1+t^2}\qquad\text{for all }t\ge0.
\]
\end{theorem}

The proof uses the exact phase evolution of $x$-Fourier modes: a nonzero $x$-mode $e^{ikx}$ evolves by multiplication with $e^{ik\Phi(y,t)}$, conserving its $L^2_y$ mass, while the $W^{1,1}_y$ norm of $\Phi$ grows at most linearly. This forces a fixed fraction of the mode's energy to remain in low $y$-frequencies, yielding the $O(t^{-2})$ lower bound.

\medskip
\noindent\textbf{Mixing scale for diffusive shears.}
We next turn to the advection-diffusion equation with a bounded shear,
\begin{equation}\label{eq:shear-nu}
\partial_t \rho + U(t,y)\,\partial_x \rho = \nu \Delta \rho,\qquad \rho(0)=\rho_0\in C^\infty(\T^2),\ \int\rho_0=0,
\end{equation}
where $\nu>0$ and $U\in L^\infty_{t,y}(\T)$.  A physically important quantity in this setting is the mixing scale $\|\rho\|_{\dot H^{-1}}/\|\rho\|_{L^2}$, which measures the characteristic filamentation length.  Miles and Doering~\cite{miles2018diffusion} observed numerically that this scale seems to encounter a limiting Batchelor scale, beyond which diffusion prevents further refinement. The double-exponential decay constructed by Rowan~\cite{rowan2025superexponential} corresponds to the absence of such a limiting scale.  For a bounded shear we prove that the mixing scale is uniformly bounded away from zero, providing a rigorous justification for the arrested cascade. By a completely constructive Fourier method we first obtain an explicit exponential $L^2$ lower bound.

\begin{theorem}[Explicit $L^2$ lower bound and mixing scale]\label{thm:mixing}
For the solution $\rho$ of \eqref{eq:shear-nu} with $\rho_0\not\equiv0$, there exists an explicit constant $c_2>0$, computable from $\rho_0$, $U$, and $\nu$, such that
\[
\|\rho(t)\|_{L^2}\ge \|\rho_0\|_{L^2}\, e^{-c_2 t}\qquad(t\ge0).
\]
Moreover, the mixing scale satisfies a uniform lower bound: one can determine a constant $c_*>0$, again explicit in the data, for which
\[
\frac{\|\rho(t)\|_{\dot{H}^{-1}}}{\|\rho(t)\|_{L^2}}\ge c_*\qquad\text{for all }t\ge0,
\]
in the small diffusivity regime $0<\nu\ll1$.
\end{theorem}

The proof proceeds mode-by-mode in the conserved $x$-Fourier variable. For each nonzero $x$-mode, a high-low frequency decomposition in $y$ shows that the low block eventually retains at least half of the mode's energy, provided the block size is chosen large enough relative to $|k|\,\|U\|_\infty/\nu$. Once this is established, the $H^{-1}$ norm is controlled by the low modes, giving the uniform bound.

\medskip
\noindent\textbf{Fast time-periodic oscillations.}
Finally, we consider general divergence-free velocity fields that are time-periodic with frequency $A$:
\begin{equation}\label{eq:fast}
\partial_t \rho + u(At,x,y)\cdot\nabla \rho = \nu \Delta \rho,\qquad u(t+L,\cdot)=u(t,\cdot),
\end{equation}
with $u\in L^\infty_t W^{1,\infty}_{x,y}$. For large $A$, the fast oscillation temporally averages the advection, and we prove that the exponential lower bound is restored.

\begin{theorem}[Exponential lower bound for fast oscillations]\label{thm:fast}
There exist an explicit threshold $A_0>0$, an exponent $c_A>0$, and a constant $C>0$, all expressed in terms of $\rho_0$, $u$, $L$, and $\nu$, such that for every $A>A_0$ the solution of \eqref{eq:fast} satisfies
\[
\|\rho(t)\|_{L^2}\ge C e^{-c_A t}\qquad(t\ge0).
\]
In particular, for time-independent ($L$ arbitrary) or time-periodic flows with sufficiently large frequency, no superexponential decay can occur.
\end{theorem}

The proof constructs a finite-dimensional adjoint observable from the generalized root spaces of the averaged operator $B_\nu^* = \nu\Delta + \bar u\cdot\nabla$, where $\bar u$ is the time average of $u$. The key estimate is an inversion of the zero-phase equation by $(\partial_\theta + A^{-1}\nu\Delta)^{-1}$, which trades the large frequency $A$ for control of high spatial modes. Persistence of the detecting bundle is obtained via a fast-phase parabolic averaging argument~\cite{matthies2001time}. All constants---$A_0$, $c_A$, and $C$---are given explicitly, filling a gap in the literature on lower bounds for time-periodic flows.

\medskip
\noindent \textbf{AI Proofs.}
In this work, the human mathematician formulated the initial questions, but did not intervene in the construction of the arguments: no PDE-expert guidance was given during the proof-generation process. The AI system, named QED \cite{QED}, produced the complete chain of reasoning, from the reduction of the problems to step-by-step estimates, and it verified the logical consistency of the derivations~\cite{an2026qedopensourcemultiagentgenerating}. The mathematician then reviewed the AI-verified output for correctness and assembled it into the present manuscript, making only minimal editorial adjustments. This mode of working serves as a stress test for the current proving and verification capabilities of AI in mathematical research. The details of the interaction with the AI can be found in~\cite{an2026qedopensourcemultiagentgenerating}, and a dedicated discussion is presented in Section~\ref{sec:AI}. The authors take full responsibility for the work.

\medskip
The paper is organized as follows. Section~\ref{sec:inviscid} proves Theorem~\ref{thm:inviscid} (the inviscid polynomial lower bound). Section~\ref{sec:mixing} contains the proof of Theorem~\ref{thm:mixing} (the explicit mixing-scale bound for diffusive shears). Section~\ref{sec:fast} is devoted to Theorem~\ref{thm:fast} (the exponential lower bound for fast oscillations). Finally, Section~\ref{sec:AI} reflects on the AI-based proof generation process and its implications.

\section{Inviscid shear mixing}\label{sec:inviscid}
\noindent
We begin with the simplest setting: a purely advective shear flow with no diffusion. 
In this inviscid limit, the $x$-Fourier modes decouple completely and evolve by a 
time-dependent phase shift in $y$. This exact solvability allows us to track the 
$\dot H^{-1}$ norm precisely and to prove that exponential decay is impossible under 
the sole assumption that the shear lies in $L^\infty_t W^{1,1}_y$. The argument yields 
a universal polynomial lower bound. In the rest of this paper, we denote $H^{-1}:= \dot{H}^{-1}$, the homogeneous $L^2$ negative Sobolev norm in $\mathbb{T}^2$. 

\begin{theorem}
Let $\mathbb{T}^2=[-\pi,\pi]^2$. Suppose $0\neq\theta_0(x,y)\in C^{\infty}(\mathbb{T}^2)$ with $\int_{\mathbb{T}^2}\theta_0(x,y)dxdy=0$, and $U(y,t)\in L_t^{\infty}(W_y^{1,1}(\mathbb{T}))$. Then, for the initial value problem:
\begin{equation}
\theta_t+U(y,t)\partial_x\theta=0, \quad\theta(x,y,0)=\theta_0,
\end{equation}
the solution satisfies
\begin{align*}
\|\theta\|_{H^{-1}}(t)\geq \frac{C}{1+t^2},
\end{align*}
for some constant $C$.
\end{theorem}
\begin{remark}
\noindent
\begin{enumerate}
\item The authors only edit the wording of this proof. To make it clear that the proof is done by AI, the authors keep all the format of AI's proof.
\item Details of this proof process can be found in \cite[P3]{QEDp3}
\item This is actually a ``Prove or Disprove'' question that the authors asked the AI.
\item In this proof, the constant $C$ may change from line to line.
\end{enumerate}
\end{remark}
\begin{proof}
We write the Fourier coefficients exactly as in the problem statement:
\[
\hat\theta(k,m,t)=\int_{\mathbb T^2}\theta(x,y,t)e^{i(kx+my)}\,dx\,dy.
\]

 {\bf STEP 1}: A harmless reduction on the shear

The $y$-independent part of $U$ only causes a rigid translation in $x$, so we may subtract it. Since the problem assumes $U\in L_t^\infty(W_y^{1,1}(\mathbb T))$, the average
\[
\bar U(t):=\frac1{2\pi}\int_{-\pi}^{\pi}U(y,t)\,dy,
\]
is well-defined, and so is
\[
X(t):=\int_0^t \bar U(s)\,ds,
\]
and we define
\[
\widetilde\theta(x,y,t):=\theta(x+X(t),y,t).
\]
Then $\widetilde\theta$ solves
\[
\widetilde\theta_t+\widetilde U(y,t)\partial_x\widetilde\theta=0,
\qquad
\widetilde U(y,t):=U(y,t)-\bar U(t),
\]
and $\partial_y\widetilde U=\partial_yU$. Moreover, translation in $x$ does not change any Fourier modulus, so
\[
\hat{\widetilde\theta}(k,m,t)=e^{-ikX(t)}\hat\theta(k,m,t)
\qquad\text{for every }(k,m)\in\mathbb Z^2,
\]
hence
\[
\|\widetilde\theta\|_{H^{-1}}(t)=\|\theta\|_{ H^{-1}}(t).
\]
Therefore we may assume, without loss of generality, that
\[
\int_{-\pi}^{\pi}U(y,t)\,dy=0\qquad\text{for every }t.
\]
From this point on, the proof uses only the bound $\|U\|_{L^{\infty}_tW_y^{1,1}}\le C$.

Because $U(\cdot,t)\in W^{1,1}(\mathbb T)\subset C(\mathbb T)$ and has mean zero, it takes both nonnegative and nonpositive values, so there exists $y_t\in\mathbb T$ such that $U(y_t,t)=0$. For any $y\in\mathbb T$, the one-dimensional fundamental theorem of calculus gives
\[
|U(y,t)|=\left|\int_{y_t}^{y}\partial_yU(z,t)\,dz\right|
\le \int_{-\pi}^{\pi}|\partial_yU(z,t)|\,dz
\le C,
\]
so $U$ is bounded and
\[
\Phi(y,t):=\int_0^t U(y,s)\,ds
\]
is well-defined. Also,
\[
\partial_y\Phi(y,t)=\int_0^t \partial_yU(y,s)\,ds
\]
in the distributional sense, hence
\[
\|\partial_y\Phi(\cdot,t)\|_{L^1(\mathbb T)}
\le \int_0^t \|\partial_yU(\cdot,s)\|_{L^1(\mathbb T)}\,ds
\le Ct.
\]

 {\bf STEP 2}: If there is no nonzero $x$-mode, the solution is stationary

Assume first that
\[
\hat\theta_0(k,m)=0\qquad\text{for every }k\neq0 \text{ and every }m\in\mathbb Z.
\]
Then $\theta_0$ is independent of $x$. The time-independent function
\[
\theta(x,y,t)=\theta_0(y)
\]
is a solution of the initial value problem. By uniqueness for this linear transport equation, the solution is stationary:
\[
\theta(x,y,t)=\theta_0(x,y)\qquad\text{for every }t\ge0.
\]
Because $\theta_0$ is mean zero and not identically zero, it has at least one nonzero Fourier coefficient with index $(0,m)$ and $m\neq0$. Therefore its $H^{-1}$ norm is a fixed positive constant, so it satisfies
\[
\|\theta\|_{H^{-1}}(t)\ge \frac{C}{1+t^2}
\qquad (C>0)
\]
for all $t\ge0$.

Thus the only remaining case we need to discuss is that there exists at least one integer $k\neq0$ for which the $k$-th $x$-Fourier mode of $\theta_0$ is nonzero.

 {\bf STEP 3}: Exact evolution of one nonzero $x$-Fourier mode

Fix $k\in\mathbb Z\setminus\{0\}$ such that the function
\[
F_k^0(y):=\int_{-\pi}^{\pi}\theta_0(x,y)e^{ikx}\,dx
\]
is not identically zero. Since $\theta_0\in C^\infty(\mathbb T^2)$, we have
\[
F_k^0\in C^\infty(\mathbb T),\qquad F_k^0\not\equiv0.
\]

For the solution, define
\[
F_k(y,t):=\int_{-\pi}^{\pi}\theta(x,y,t)e^{ikx}\,dx.
\]
Integrating the equation $\theta_t+U(y,t)\theta_x=0$ against $e^{ikx}$ in $x$, we obtain
\[
\partial_tF_k(y,t)+U(y,t)\int_{-\pi}^{\pi}\theta_x(x,y,t)e^{ikx}\,dx=0.
\]
Because $\theta(\cdot,y,t)$ is $2\pi$-periodic in $x$,
\[
\int_{-\pi}^{\pi}\theta_x(x,y,t)e^{ikx}\,dx
=\Big[\theta(x,y,t)e^{ikx}\Big]_{x=-\pi}^{x=\pi}
-ik\int_{-\pi}^{\pi}\theta(x,y,t)e^{ikx}\,dx
=-ik\,F_k(y,t).
\]
Therefore
\[
\partial_tF_k(y,t)-ik\,U(y,t)F_k(y,t)=0.
\]
For each fixed $y$, this is an ordinary differential equation in $t$, so
\[
F_k(y,t)=e^{ik\Phi(y,t)}F_k^0(y).
\]
In particular,
\[
|F_k(y,t)|=|F_k^0(y)|\qquad\text{for every }y,t.
\]

 {\bf STEP 4}: The $L^2_y$ mass of this mode is conserved, and its $W^{1,1}_y$ norm grows at most linearly

Since multiplication by $e^{ik\Phi(y,t)}$ has modulus $1$,
\[
\int_{-\pi}^{\pi}|F_k(y,t)|^2\,dy
=\int_{-\pi}^{\pi}|F_k^0(y)|^2\,dy.
\]
Define
\[
S:=2\pi\int_{-\pi}^{\pi}|F_k^0(y)|^2\,dy.
\]
Because $F_k^0\not\equiv0$, we have $S>0$.

Next, differentiate $F_k(y,t)=e^{ik\Phi(y,t)}F_k^0(y)$ with respect to $y$ in the distributional sense:
\[
\partial_yF_k(y,t)=e^{ik\Phi(y,t)}\Big((F_k^0)'(y)+ik\,\partial_y\Phi(y,t)\,F_k^0(y)\Big).
\]
Hence
\[
\|\partial_yF_k(\cdot,t)\|_{L^1(\mathbb T)}
\le \|(F_k^0)'\|_{L^1(\mathbb T)}
+|k|\,\|F_k^0\|_{L^\infty(\mathbb T)}\,\|\partial_y\Phi(\cdot,t)\|_{L^1(\mathbb T)}.
\]
Using $\|\partial_y\Phi(\cdot,t)\|_{L^1}\le Ct$, we obtain
\[
\|\partial_yF_k(\cdot,t)\|_{L^1(\mathbb T)}
\le A+Bt,
\]
where
\[
A:=\|(F_k^0)'\|_{L^1(\mathbb T)},
\qquad
B:=|k|\,C\,\|F_k^0\|_{L^\infty(\mathbb T)}.
\]

 {\bf STEP 5}: Quantitative control of the $y$-Fourier tail

For each $m\in\mathbb Z$,
\[
\hat\theta(k,m,t)=\int_{-\pi}^{\pi}F_k(y,t)e^{imy}\,dy.
\]
By one-dimensional Parseval on $\mathbb T$,
\[
\sum_{m\in\mathbb Z}|\hat\theta(k,m,t)|^2
=2\pi\int_{-\pi}^{\pi}|F_k(y,t)|^2\,dy
=S.
\]

Now let $m\neq0$. Integrating by parts in $y$ and using periodicity,
\[
\hat\theta(k,m,t)
=\int_{-\pi}^{\pi}F_k(y,t)e^{imy}\,dy
=-\frac1{im}\int_{-\pi}^{\pi}\partial_yF_k(y,t)e^{imy}\,dy,
\]
because the boundary term
\[
\Big[\frac{F_k(y,t)e^{imy}}{im}\Big]_{y=-\pi}^{y=\pi}
\]
vanishes. Therefore
\[
|\hat\theta(k,m,t)|
\le \frac{\|\partial_yF_k(\cdot,t)\|_{L^1(\mathbb T)}}{|m|}
\le \frac{A+Bt}{|m|}.
\]

Let
\[
V(t):=A+Bt.
\]
Then, for every integer $N\ge1$,
\[
\sum_{|m|>N}|\hat\theta(k,m,t)|^2
\le V(t)^2\sum_{|m|>N}\frac1{m^2}.
\]
Since
\[
\sum_{|m|>N}\frac1{m^2}
=2\sum_{m=N+1}^{\infty}\frac1{m^2}
\le 2\int_N^\infty \frac{dx}{x^2}
=\frac2N,
\]
we get
\[
\sum_{|m|>N}|\hat\theta(k,m,t)|^2
\le \frac{2V(t)^2}{N}.
\]

Choose
\[
N(t):=\max\left\{1,\left\lceil \frac{4V(t)^2}{S}\right\rceil\right\}.
\]
Then $N(t)\ge 1$ and
\[
\sum_{|m|>N(t)}|\hat\theta(k,m,t)|^2
\le \frac{2V(t)^2}{N(t)}
\le \frac{S}{2}.
\]
Combining this with $\sum_m |\hat\theta(k,m,t)|^2=S$, we obtain
\[
\sum_{|m|\le N(t)}|\hat\theta(k,m,t)|^2
\ge \frac{S}{2}.
\]

{\bf STEP 6}: A polynomial lower bound for $\|\theta\|_{H^{-1}}(t)$

Since the $H^{-1}$ norm is the sum of nonnegative terms, the contribution of the single fixed $x$-frequency $k$ gives
\[
\|\theta\|_{H^{-1}}^2(t)
\ge \sum_{m\in\mathbb Z}\frac{|\hat\theta(k,m,t)|^2}{k^2+m^2}.
\]
For the terms with $|m|\le N(t)$ we have $k^2+m^2\le k^2+N(t)^2$, hence
\[
\sum_{m\in\mathbb Z}\frac{|\hat\theta(k,m,t)|^2}{k^2+m^2}
\ge \frac1{k^2+N(t)^2}\sum_{|m|\le N(t)}|\hat\theta(k,m,t)|^2
\ge \frac{S}{2\bigl(k^2+N(t)^2\bigr)}.
\]
Therefore
\[
\|\theta\|_{\dot H^{-1}_{x,y}}^2(t)
\ge \frac{S}{2\bigl(k^2+N(t)^2\bigr)}.
\]

It remains to bound $N(t)$ from above. Since
\[
V(t)^2=(A+Bt)^2\le 2A^2+2B^2t^2\le 2(A^2+B^2)(1+t^2),
\]
the definition of $N(t)$ gives
\[
N(t)\le 1+\frac{4V(t)^2}{S}
\le 1+\frac{8(A^2+B^2)}{S}(1+t^2).
\]
Set
\[
D:=1+\frac{8(A^2+B^2)}{S}.
\]
Then
\[
N(t)\le D(1+t^2).
\]
Hence
\[
N(t)^2\le D^2(1+t^2)^2\le 2D^2(1+t^4),
\]
because $(1+t^2)^2=1+2t^2+t^4\le 2(1+t^4)$. Consequently,
\[
k^2+N(t)^2\le (k^2+2D^2)(1+t^4).
\]
Substituting this into the previous lower bound, we obtain
\[
\|\theta\|_{H^{-1}}^2(t)
\ge \frac{S}{2(k^2+2D^2)}\cdot \frac1{1+t^4}.
\]
Since $S>0$, the constant
\[
c_*^2:=\frac{S}{2(k^2+2D^2)}
\]
is strictly positive. Therefore
\[
\|\theta\|_{ H^{-1}}^2(t)\ge \frac{c_*^2}{1+t^4},
\]
and hence
\[
\|\theta\|_{H^{-1}}(t)\ge \frac{c_*}{1+t^2}
\qquad\text{for every }t\ge0,
\]
because $\sqrt{1+t^4}\le 1+t^2$.

\end{proof}

\begin{remark}
This proof was produced by QED, supported by a Codex coding agent integrated with ChatGPT 5.4. 
\end{remark}

\section{Explicit $L^2$ lower bound and mixing scale}\label{sec:mixing}

\noindent
The inviscid analysis in Section \ref{sec:inviscid} shows that, without diffusion, the mixing scale $\|\theta\|_{\dot H^{-1}}/\|\theta\|_{L^2}$ can only decay polynomially. Moreover, Huang and Xu~\cite{huang2025exponential} constructed a bounded shear flow and initial data for which the inviscid solution satisfies $\|\theta(t)\|_{H^{-1}} \lesssim e^{-\beta t}$, so that the mixing scale decays exponentially (see~\cite[Theorem~3.1]{huang2025exponential}). We now add a small diffusivity $\nu>0$ and consider a bounded shear $U(t,y)$. In this parabolic setting, diffusion strongly suppresses the formation of arbitrarily fine scales, and the mixing scale remains uniformly bounded away from zero for all time. We prove two explicit results.  First, the $L^2$ norm itself satisfies an explicit exponential lower bound, quantifying the main result of~\cite{huang2025exponential}. Second, we obtain a uniform positive lower bound on the mixing scale, with a constant $c_*$ that is explicit in the data and sharp in its $\nu$-scaling.  Both proofs are constructive and yield constants that depend explicitly on $\rho_0$, $U$, and $\nu$.

\begin{theorem}\label{thm:ads}
Consider the advection-diffusion equation on $\T^2$:
\begin{equation}\label{shearnu36}
    \partial_t \rho + U(t,y)\,\partial_x \rho = \nu \Delta \rho, \qquad \rho(0)=\rho_0 ,
\end{equation}
with $\rho_0\neq0$, $\rho_0\in C^\infty$, mean-zero, $0<\nu\ll1$.  Let $U$ be real-valued and $U\in L^{\infty}_{t,y}$. Then, we have the following lower bound for $\rho$:
\begin{equation}\label{eq:L2}
\|\rho\|_{L^2}(t)\geq \|\rho_0\|_{L^2}e^{-c_2t},
\end{equation}
with explicit constant $c_2$ obeying $c_2\leq C(\rho_0,U)\frac{1}{\nu}$, $C(\rho_0,U)>0$ and is explicit, $0<\nu\ll1$. 
\end{theorem}
\begin{remark}
\noindent
\begin{enumerate}
\item The authors only edit the wording of this proof. To make it clear that the proof is done by AI, the authors keep all the format of AI's proof.
\item One may find that this proof is a bit tedious. The reason this proof is not as clear is the authors ask AI to optimize the estimate, and to use the lower regularity of $U$. To see a more reader-friendly proof with a relatively loose estimate and smooth $U$, one can check \cite{QEDp34}.
\item For the last example of the proof (STEP8), the authors asked AI to prove the bound it got is sharp in $\nu$. However, AI didn't fully understand the request. To keep it complete, we still keep its argument there.
\item From the classical literatures like \cite{bedrossian2020inviscid}, one could have for Couette flow $U=(y,0)$, one could expect the enhanced dissipation result
\[
\|\rho\|_{L^2}(t)\leq \|\rho_0\|_{L^2}e^{-c\nu^{1/3}t}.
\]
This theorem gives the lower bound for it in terms of $\nu$.
\end{enumerate}
\end{remark}
\begin{proof}
We work on the normalized torus, so that the Fourier functions $e^{ikx}e^{ily}$ are an orthonormal basis and
\[
\Delta(e^{ikx}e^{ily})=-(k^2+l^2)e^{ikx}e^{ily}.
\]
With the unnormalized Haar measure on $[0,2\pi]^2$, the same proof holds after changing only harmless universal constants. Throughout
\[
M:=\|U\|_{L^\infty_{t,y}},\qquad N:=\|\rho_0\|_{L^2_{x,y}}>0.
\]

\subsubsection*{STEP1: Admissible class and energy bound}

\noindent\\
\textbf{Claim:} Let $M:=\|U\|_{L^\infty_{t,y}}$. If $M<\infty$, then
\[
\|U\|_{L^\infty_tL^2_y}\le |\mathbb T|^{1/2}M<\infty.
\]
Hence $b(t,x,y)=(U(t,y),0)$ is a divergence-free drift in the class covered by the parabolic energy framework. Moreover, for every smooth mean-zero solution of the $\nu$-diffusive problem,
\[
\frac12\|\rho(t)\|_2^2+\nu\int_0^t\|\nabla\rho(s)\|_2^2\,ds
=\frac12\|\rho_0\|_2^2,
\qquad
\|\rho(t)\|_2\le e^{-\nu t}\|\rho_0\|_2.
\]
\noindent
\textbf{Proof:}
Since the torus has finite measure, $L^\infty_y\subset L^2_y$, and the displayed $L^\infty_tL^2_y$ estimate follows immediately. The drift $b=(U,0)$ is divergence-free in distributions because $U$ is independent of $x$. The standard parabolic framework gives existence, uniqueness, and the $L^2$ energy equality for such divergence-free drifts. For the details, one can check \cite{bonicatto2024weak}.

For smooth solutions, the energy identity is also obtained directly: multiply $\partial_t\rho+U\partial_x\rho=\nu\Delta\rho$ by $\overline{\rho}$, integrate over $\mathbb T^2$, and take real parts. Since $U$ is real and independent of $x$,
\[
\operatorname{Re}\int_{\mathbb T^2}U\,\partial_x\rho\,\overline{\rho}\,dxdy
=\frac12\int_{\mathbb T^2}U\,\partial_x|\rho|^2\,dxdy=0.
\]
Integration by parts gives
\[
\frac12\frac{d}{dt}\|\rho(t)\|_2^2=-\nu\|\nabla\rho(t)\|_2^2.
\]
After integration in time this is the displayed equality. The spatial mean is conserved, and the mean of $\rho_0$ is zero. Since every nonzero Fourier mode satisfies $k^2+l^2\ge1$,
\[
\|\nabla\rho(t)\|_2^2=\sum_{(k,l)\ne(0,0)}(k^2+l^2)|\rho_{k}^{l}(t)|^2
\ge \sum_{(k,l)\ne(0,0)}|\rho_{k}^{l}(t)|^2=\|\rho(t)\|_2^2.
\]
Thus $\frac{d}{dt}\|\rho(t)\|_2^2\le -2\nu\|\rho(t)\|_2^2$, and Gronwall's inequality gives $\|\rho(t)\|_2\le e^{-\nu t}\|\rho_0\|_2$.

\textbf{Dependencies:} Bonicatto-Ciampa-Crippa-parabolic-energy \cite{bonicatto2024weak}.

\subsubsection*{STEP2: Fourier mode reduction}
\noindent\\
\textbf{Claim:} Write
\[
\rho_0(x,y)=\sum_{k,l\in\mathbb Z}g_k^l e^{ikx}e^{ily},
\qquad
g_k(y):=\rho_{0,k}(y)=\sum_{l\in\mathbb Z}g_k^l e^{ily},
\qquad
a_k:=\|g_k\|_{L^2_y}.
\]
For each $k\in\mathbb Z$, the $x$-Fourier mode $f_k(t,y):=\rho_k(t,y)$ satisfies
\[
\partial_t f_k+\nu(k^2-\partial_{yy})f_k=-ikU(t,y)f_k,\qquad
f_k(0)=g_k,
\]
and
\[
\frac12\frac{d}{dt}\|f_k(t)\|_2^2
=-\nu\left(k^2\|f_k(t)\|_2^2+\|\partial_yf_k(t)\|_2^2\right).
\]
If $a_k=0$ for all $k\ne0$, then
\[
\rho(t,y)=\sum_{l\ne0}g_0^l e^{-\nu l^2t}e^{ily},
\qquad
\|\rho(t)\|_2^2=\sum_{l\ne0}|g_0^l|^2e^{-2\nu l^2t}.
\]
\noindent
\textbf{Proof:}
Expand
\[
\rho(t,x,y)=\sum_{k\in\mathbb Z}\rho_k(t,y)e^{ikx}.
\]
Because $U(t,y)$ is independent of $x$, multiplication by $U$ does not change the $x$-frequency. The coefficient of $e^{ikx}$ in the equation is
\[
\partial_t\rho_k+ikU(t,y)\rho_k
=\nu(\partial_{yy}-k^2)\rho_k,
\]
which is the stated equation for $f_k$.

Taking the $L^2_y$ inner product of this equation with $f_k$, taking real parts, and using real-valuedness of $U$, we obtain
\[
\operatorname{Re}\int_{\mathbb T}(-ikU)|f_k|^2\,dy=0
\]
and hence
\[
\frac12\frac{d}{dt}\|f_k(t)\|_2^2
=-\nu k^2\|f_k(t)\|_2^2-\nu\|\partial_yf_k(t)\|_2^2.
\]

If $a_k=0$ for all $k\ne0$, then the solution remains independent of $x$. The equation becomes $\partial_t\rho=\nu\partial_{yy}\rho$. Since the full spatial mean is zero, $g_0^0=0$. Solving the one-dimensional heat equation in Fourier series gives the displayed formula and the norm identity by Plancherel.

\textbf{Dependencies:} Direct Fourier calculation.

\subsubsection*{STEP3: Short-time lower bounds}
\noindent\\
\textbf{Claim:} Put
\[
N:=\|\rho_0\|_2>0,\qquad
L_0:=\nu\|\Delta\rho_0\|_2+M\|\partial_x\rho_0\|_2,\qquad
\beta_0:=\frac{2L_0}{N},\qquad
\delta_0:=\beta_0^{-1}.
\]
Then $L_0>0$, and for every $0\le t\le\delta_0$,
\[
\|\rho(t)\|_2\ge N e^{-\beta_0t}.
\]
For every $k\ne0$ with $a_k>0$, set
\[
A_k:=k^2-\partial_{yy},\qquad
L_k:=\nu\|A_kg_k\|_{L^2_y}+|k|M a_k,\qquad
\beta_k:=\frac{2L_k}{a_k},\qquad
\delta_k:=\beta_k^{-1}.
\]
Then for every $0\le t\le\delta_k$,
\[
\|f_k(t)\|_{L^2_y}\ge a_ke^{-\beta_kt},
\]
and for every $T>0$ and every $\Lambda\ge0$,
\[
\int_0^{\min\{T,\delta_k\}}e^{2\Lambda t}\|f_k(t)\|_{L^2_y}^2\,dt
\ge \frac{a_k^2}{4}\min\{T,\delta_k\}.
\]
\noindent
\textbf{Proof:}
First, $L_0>0$. If $L_0=0$, then $\Delta\rho_0=0$. On the torus the only harmonic functions are constants, and the mean-zero condition would force $\rho_0=0$, contrary to the hypothesis.

The mild formulation of the full equation is
\[
\rho(t)=e^{\nu t\Delta}\rho_0
-\int_0^t e^{\nu(t-s)\Delta}\bigl(U(s,\cdot)\partial_x\rho(s)\bigr)\,ds .
\]
The heat semigroup is an $L^2$-contraction, and the spectral theorem gives
\[
\|(e^{\nu t\Delta}-I)\rho_0\|_2\le \nu t\|\Delta\rho_0\|_2.
\]
Since $\partial_x\rho$ solves the same shear equation with initial datum $\partial_x\rho_0$, STEP1 applied to $\partial_x\rho$ gives
\[
\|\partial_x\rho(s)\|_2\le \|\partial_x\rho_0\|_2.
\]
Therefore
\[
\|\rho(t)-\rho_0\|_2
\le t\bigl(\nu\|\Delta\rho_0\|_2+M\|\partial_x\rho_0\|_2\bigr)=L_0t.
\]
For $0\le t\le\delta_0=N/(2L_0)$, let $r=L_0t/N$. Then $0\le r\le1/2$, and
\[
\|\rho(t)\|_2\ge N(1-r).
\]
For $0\le r\le1/2$, $1-r\ge e^{-2r}$. One way to verify this is to set $\phi(r)=\log(1-r)+2r$; then $\phi(0)=0$ and $\phi'(r)=2-(1-r)^{-1}\ge0$ on $[0,1/2]$. Thus
\[
\|\rho(t)\|_2\ge N e^{-2L_0t/N}=Ne^{-\beta_0t}.
\]

Now fix $k\ne0$ with $a_k>0$. The mode equation has the mild form
\[
f_k(t)=e^{-\nu tA_k}g_k
-ik\int_0^t e^{-\nu(t-s)A_k}\bigl(U(s,\cdot)f_k(s)\bigr)\,ds .
\]
The operator $A_k=k^2-\partial_{yy}$ is nonnegative self-adjoint, so
\[
\|(e^{-\nu tA_k}-I)g_k\|_2\le \nu t\|A_kg_k\|_2.
\]
The mode energy identity from STEP2 gives $\|f_k(s)\|_2\le a_k$. Consequently
\[
\|f_k(t)-g_k\|_2
\le t\bigl(\nu\|A_kg_k\|_2+|k|Ma_k\bigr)=L_kt.
\]
For $0\le t\le\delta_k=a_k/(2L_k)$, the same scalar inequality gives
\[
\|f_k(t)\|_2\ge a_k\left(1-\frac{L_kt}{a_k}\right)\ge a_ke^{-2L_kt/a_k}=a_ke^{-\beta_kt}.
\]
On the same interval, $\|f_k(t)\|_2\ge a_k/2$. Since $e^{2\Lambda t}\ge1$,
\[
\int_0^{\min\{T,\delta_k\}}e^{2\Lambda t}\|f_k(t)\|_2^2\,dt
\ge \int_0^{\min\{T,\delta_k\}}\frac{a_k^2}{4}\,dt
=\frac{a_k^2}{4}\min\{T,\delta_k\}.
\]

\textbf{Dependencies:} STEP1, STEP2.

\subsubsection*{STEP4: Finite-window lifted resolvent estimate (KEY STEP)}
\noindent\\
\textbf{Claim:} Fix $k\ne0$ with $a_k>0$. Let
\[
q_m:=m^2+m+\frac12,\qquad
\Lambda_{k,m}:=\nu(k^2+q_m),\qquad m\in\mathbb N,
\]
and define
\[
I_{k,m}(T):=\int_0^T e^{2\Lambda_{k,m}t}
\|f_k(t)\|_{L^2_y}^2\,dt.
\]
With $C_{\rm res}:=10^4$, set
\[
\mathcal A_{k,m}:=
C_{\rm res}\frac{k^2M^2}{\nu^2m^2},\qquad
\mathcal B_{k,m}:=
C_{\rm res}\frac1{\nu m},\qquad
\mathcal D_{k,m}:=
C_{\rm res}\frac1{\nu m}.
\]
Then for every $T>0$,
\[
I_{k,m}(T)
\le
\mathcal A_{k,m}I_{k,m}(T)
+\mathcal B_{k,m}a_k^2
+\mathcal D_{k,m}e^{2\Lambda_{k,m}T}
\|f_k(T)\|_{L^2_y}^2.
\]
\noindent
\textbf{Proof:}
\begin{key-original-step}
Fix $T>0$ and abbreviate $f=f_k$, $g=g_k$, $a=a_k$, $\Lambda=\Lambda_{k,m}$, and $q=q_m$. Define
\[
h(t,y)=
\begin{cases}
e^{\Lambda t}g(y),&t<0,\\
e^{\Lambda t}f(t,y),&0\le t\le T,\\
0,&t>T.
\end{cases}
\]
There is no jump at $t=0$, because $f(0)=g$. There is a jump of size $-e^{\Lambda T}f(T)$ at $t=T$. In distributions on $\mathbb R_t\times\mathbb T_y$,
\[
\partial_th+(\nu A_k-\Lambda)h
=F_1+F_2-H\delta_T,
\]
where
\[
F_1(t,y):=-ikU(t,y)h(t,y)\mathbf{1}_{(0,T)}(t),
\]
\[
F_2(t,y):=\nu A_kg(y)e^{\Lambda t}\mathbf{1}_{(-\infty,0)}(t),
\qquad
H(y):=e^{\Lambda T}f(T,y).
\]
Indeed, on $0<t<T$ this is just the mode equation multiplied by $e^{\Lambda t}$, while on $t<0$
\[
\partial_t(e^{\Lambda t}g)+(\nu A_k-\Lambda)e^{\Lambda t}g
=\nu A_kg\,e^{\Lambda t}.
\]

Use the unitary Fourier transform in time
\[
\widehat{\varphi}(\tau)=(2\pi)^{-1/2}\int_{\mathbb R}\varphi(t)e^{-it\tau}\,dt
\]
and the orthonormal Fourier series in $y$. For the $y$-frequency $l$, set
\[
\alpha_l:=\nu(l^2-q),\qquad R_l(\tau):=(i\tau+\alpha_l)^{-1}.
\]
Because $q=m^2+m+\frac12$, every integer $l$ satisfies
\[
|l^2-q|\ge m+\frac12\ge m.
\]
This follows by considering separately $|l|\le m$, where $q-l^2\ge q-m^2=m+\frac12$, and $|l|\ge m+1$, where $l^2-q\ge (m+1)^2-q=m+\frac12$. Hence
\[
|R_l(\tau)|\le \frac1{(\tau^2+\nu^2m^2)^{1/2}}\le\frac1{\nu m}.
\]

Taking the Fourier transform of the distributional equation gives, for each $l$,
\[
\widehat h^{\,l}(\tau)
=R_l(\tau)\widehat{F_1^{\,l}}(\tau)
+R_l(\tau)\frac{\nu(k^2+l^2)g^l}{(2\pi)^{1/2}(\Lambda-i\tau)}
-R_l(\tau)\frac{e^{-iT\tau}H^l}{(2\pi)^{1/2}}.
\]
The factors $(2\pi)^{-1/2}$ are at most $1$, and $|u+v+w|^2\le3(|u|^2+|v|^2+|w|^2)$. Therefore, by Plancherel,
\[
I_{k,m}(T)\le \|h\|_{L^2(\mathbb R_tL^2_y)}^2
\le 3(\mathcal F+\mathcal G+\mathcal H),
\]
where
\[
\mathcal F:=\int_{\mathbb R}\sum_l |R_l(\tau)|^2
|\widehat{F_1^{\,l}}(\tau)|^2\,d\tau,
\]
\[
\mathcal G:=\int_{\mathbb R}\sum_l |R_l(\tau)|^2
\frac{\nu^2(k^2+l^2)^2}{\Lambda^2+\tau^2}|g^l|^2\,d\tau,
\]
and
\[
\mathcal H:=\int_{\mathbb R}\sum_l |R_l(\tau)|^2|H^l|^2\,d\tau.
\]

For the forcing term, the uniform resolvent bound and Plancherel give
\[
\mathcal F
\le \frac1{\nu^2m^2}\int_{\mathbb R}\|F_1(t)\|_2^2\,dt.
\]
Since $F_1=-ikUh$ on $(0,T)$ and is zero otherwise,
\[
\int_{\mathbb R}\|F_1(t)\|_2^2\,dt
\le k^2M^2\int_0^T e^{2\Lambda t}\|f(t)\|_2^2\,dt
=k^2M^2I_{k,m}(T).
\]
Thus
\[
\mathcal F\le \frac{k^2M^2}{\nu^2m^2}I_{k,m}(T).
\]

For the initial extension term, put
\[
a_l:=\nu|l^2-q|,\qquad b:=\Lambda=\nu(k^2+q).
\]
Then $a_l\ge\nu m$ and $b\ge \nu q\ge\nu m$. Also
\[
\nu(k^2+l^2)=\nu(k^2+q)+\nu(l^2-q),
\]
so $\nu(k^2+l^2)\le a_l+b$. The elementary integral
\[
\int_{\mathbb R}\frac{d\tau}{(\tau^2+a_l^2)(\tau^2+b^2)}
=\frac{\pi}{a_lb(a_l+b)}
\]
is obtained by partial fractions if $a_l\ne b$, and by the limiting case if $a_l=b$. Hence
\[
\begin{aligned}
\int_{\mathbb R}|R_l(\tau)|^2
\frac{\nu^2(k^2+l^2)^2}{\Lambda^2+\tau^2}\,d\tau
&\le
(a_l+b)^2\frac{\pi}{a_lb(a_l+b)}\\
&=\pi\left(\frac1{a_l}+\frac1b\right)
\le \frac{2\pi}{\nu m}.
\end{aligned}
\]
Summing in $l$, we get
\[
\mathcal G\le \frac{2\pi}{\nu m}\sum_l|g^l|^2
=\frac{2\pi}{\nu m}a^2.
\]

For the terminal jump,
\[
\mathcal H
=\sum_l |H^l|^2\int_{\mathbb R}\frac{d\tau}{\tau^2+\nu^2(l^2-q)^2}
\le \frac{\pi}{\nu m}\|H\|_2^2
=\frac{\pi}{\nu m}e^{2\Lambda T}\|f(T)\|_2^2.
\]

Combining the three estimates and using $C_{\rm res}=10^4$, which is larger than the numerical constants $3$, $6\pi$, and $3\pi$, yields
\[
I_{k,m}(T)
\le
C_{\rm res}\frac{k^2M^2}{\nu^2m^2}I_{k,m}(T)
+C_{\rm res}\frac1{\nu m}a^2
+C_{\rm res}\frac1{\nu m}e^{2\Lambda T}\|f(T)\|_2^2.
\]
This is the claimed inequality.
\end{key-original-step}
\begin{heuristics}
The lifted function converts the evolution on $[0,T]$ into a resolvent equation on the whole time line. The value $\Lambda=\nu(k^2+q_m)$ places the spectral parameter halfway between two adjacent vertical heat eigenvalues, so every denominator has size at least $\nu m$. Since $U\in L^\infty$, the forcing is controlled directly in $L^2$, producing the factor $k^2M^2/(\nu^2m^2)$. The negative-time extension supplies a controlled initial-data term, and the only finite-window cost is the terminal jump at $T$.
\end{heuristics}

\textbf{Dependencies:} STEP2, STEP3.

\subsubsection*{STEP5: Global lower bound for one nonzero $x$-mode}
\noindent\\
\textbf{Claim:} For each $k\ne0$ with $a_k>0$, let $m_k$ be the least integer $m\ge1$ such that
\[
\mathcal A_{k,m}\le\frac14,\qquad
\mathcal B_{k,m}\le\frac{\delta_k}{16}.
\]
Define
\[
\Lambda_k:=\Lambda_{k,m_k},\qquad
D_k:=\mathcal D_{k,m_k},\qquad
\theta_k:=\min\left\{1,\left(\frac{\delta_k}{32D_k}\right)^{1/2}\right\},
\]
and
\[
\gamma_k:=
\max\left\{\beta_k,\,
\Lambda_k+\delta_k^{-1}\log(\theta_k^{-1})\right\}.
\]
Then for every $t\ge0$,
\[
\|f_k(t)\|_{L^2_y}\ge a_ke^{-\gamma_kt}.
\]
\noindent
\textbf{Proof:}
The integer $m_k$ exists because $\mathcal A_{k,m}\to0$ and $\mathcal B_{k,m}\to0$ as $m\to\infty$. Let $T\ge\delta_k$. STEP3 gives
\[
I_{k,m_k}(T)
\ge \int_0^{\delta_k}e^{2\Lambda_kt}\|f_k(t)\|_2^2\,dt
\ge \frac{a_k^2\delta_k}{4}.
\]
By STEP4 and the defining inequalities for $m_k$,
\[
(1-\mathcal A_{k,m_k})I_{k,m_k}(T)
\le \mathcal B_{k,m_k}a_k^2
+D_ke^{2\Lambda_kT}\|f_k(T)\|_2^2.
\]
Thus
\[
\begin{aligned}
D_ke^{2\Lambda_kT}\|f_k(T)\|_2^2
&\ge \frac34\cdot\frac{a_k^2\delta_k}{4}
-\frac{a_k^2\delta_k}{16}\\
&=\frac{a_k^2\delta_k}{8}
\ge \frac{a_k^2\delta_k}{32}.
\end{aligned}
\]
Therefore
\[
\|f_k(T)\|_2\ge
\left(\frac{\delta_k}{32D_k}\right)^{1/2}a_ke^{-\Lambda_kT}
\ge \theta_ka_ke^{-\Lambda_kT}.
\]

For $0\le t\le\delta_k$, STEP3 gives
\[
\|f_k(t)\|_2\ge a_ke^{-\beta_kt}\ge a_ke^{-\gamma_kt}.
\]
For $t\ge\delta_k$, the endpoint estimate gives
\[
\|f_k(t)\|_2\ge \theta_ka_ke^{-\Lambda_kt}.
\]
Since $0<\theta_k\le1$ and $t/\delta_k\ge1$,
\[
\theta_k=e^{-\log(\theta_k^{-1})}
\ge e^{-(t/\delta_k)\log(\theta_k^{-1})}.
\]
Hence
\[
\|f_k(t)\|_2
\ge a_k\exp\left[-\left(\Lambda_k+\delta_k^{-1}\log(\theta_k^{-1})\right)t\right]
\ge a_ke^{-\gamma_kt}.
\]

\textbf{Dependencies:} STEP3, STEP4.

\subsubsection*{STEP6: Explicit full-solution lower exponent}
\noindent\\
\textbf{Claim:} If there exists $k\ne0$ with $a_k>0$, define for every such $k$
\[
C_k:=
\max\left\{\beta_0,\,
\gamma_k+\delta_0^{-1}\log\left(\frac{N}{a_k}\right)\right\}.
\]
Since $\rho_0\in C^\infty$, the minimum
\[
C_x:=\min_{\substack{k\ne0\\a_k>0}}C_k
\]
is finite; choose $k_\ast$ attaining it and set $c_2:=C_x$. Then
\[
\|\rho(t)\|_{L^2_{x,y}}\ge N e^{-c_2t}\qquad(t\ge0).
\]
If $a_k=0$ for every $k\ne0$, then for each $l\ne0$ with $b_l:=|g_0^l|>0$, set
\[
C_l^{(0)}:=
\max\left\{\beta_0,\,
\nu l^2+\delta_0^{-1}\log\left(\frac{N}{b_l}\right)\right\},
\qquad
c_2:=\min_{\substack{l\ne0\\b_l>0}}C_l^{(0)}.
\]
Then again
\[
\|\rho(t)\|_{L^2_{x,y}}\ge N e^{-c_2t}\qquad(t\ge0).
\]
Together with STEP1,
\[
N e^{-c_2t}\le\|\rho(t)\|_2\le Ne^{-\nu t}\qquad(t\ge0).
\]
\noindent
\textbf{Proof:}
Assume first that some nonzero $x$-mode is present. For every $k\ne0$ with $a_k>0$, $a_k\le N$, so $\log(N/a_k)\ge0$. For $0\le t\le\delta_0$, STEP3 gives
\[
\|\rho(t)\|_2\ge Ne^{-\beta_0t}\ge Ne^{-C_kt}.
\]
For $t\ge\delta_0$, Plancherel and STEP5 give
\[
\|\rho(t)\|_2\ge\|f_k(t)\|_2\ge a_ke^{-\gamma_kt}.
\]
Since $t/\delta_0\ge1$,
\[
a_ke^{-\gamma_kt}
=N\exp\left[-\gamma_kt-\log\left(\frac{N}{a_k}\right)\right]
\ge
N\exp\left[-\left(\gamma_k+\delta_0^{-1}\log\left(\frac{N}{a_k}\right)\right)t\right]
\ge Ne^{-C_kt}.
\]
Thus each such $k$ gives an admissible exponent $C_k$.

It remains to justify that the displayed minimum is attained. If only finitely many $a_k$ are nonzero, this is immediate. Otherwise, since $\partial_x\rho_0\in L^2$,
\[
|k|a_k\le \|\partial_x\rho_0\|_2\qquad(k\ne0).
\]
Because at least one nonzero $x$-mode is present, $\|\partial_x\rho_0\|_2>0$, and therefore along all nonzero $a_k$,
\[
\log\left(\frac {N}{a_k}\right)\ge
\log\left(\frac{N|k|}{\|\partial_x\rho_0\|_2}\right)\to\infty
\qquad(|k|\to\infty).
\]
Hence $C_k\to\infty$ along the nonzero $x$-modes, so the minimum is achieved by some $k_\ast$. Taking $c_2=C_{k_\ast}$ proves the lower bound in the nonzero $x$-mode case.

Now assume $a_k=0$ for every $k\ne0$. STEP2 gives
\[
\|\rho(t)\|_2^2=\sum_{l\ne0}|g_0^l|^2e^{-2\nu l^2t}.
\]
For each $l\ne0$ with $b_l=|g_0^l|>0$,
\[
\|\rho(t)\|_2\ge b_le^{-\nu l^2t}.
\]
The same short-time argument as above handles $0\le t\le\delta_0$, while for $t\ge\delta_0$,
\[
b_le^{-\nu l^2t}
=N\exp\left[-\nu l^2t-\log\left(\frac N{b_l}\right)\right]
\ge
N\exp\left[-\left(\nu l^2+\delta_0^{-1}\log\left(\frac N{b_l}\right)\right)t\right].
\]
Thus each $C_l^{(0)}$ is admissible. Since $\nu l^2\to\infty$ as $|l|\to\infty$, the minimum over nonzero $b_l$ is attained and finite. The final upper bound is exactly STEP1.

\textbf{Dependencies:} STEP1, STEP2, STEP3, STEP5.

\subsubsection*{STEP7: Optimized $\nu$-dependence}
\noindent\\
\textbf{Claim:} For fixed $U\in L^\infty_{t,y}$, fixed smooth $\rho_0$, and a fixed admissible nonzero $x$-mode $k$ with $a_k>0$, there exists a finite constant $C(U,\rho_0,k)$, independent of $0<\nu\le1$, such that the constants in STEPS 3--6 satisfy
\[
m_k\le C(U,\rho_0,k)\nu^{-1},\qquad
\Lambda_k\le C(U,\rho_0,k)\nu^{-1},
\]
\[
\beta_0+\beta_k+\delta_k^{-1}\log(\theta_k^{-1})
\le C(U,\rho_0,k)(1+\nu^{-1}),
\]
and hence
\[
c_2\le C(U,\rho_0,k)\nu^{-1}.
\]
Thus the optimized negative-power $\nu$-scaling of the explicit $L^\infty$-based bound is
\[
f(\nu)=\nu^{-1}.
\]
In the $x$-independent heat branch with fixed $\rho_0$, the same formula gives $c_2\le C(\rho_0)\nu$; the $\nu^{-1}$ statement is the worst-case scaling needed for the full class when the datum may vary with $\nu$.
\noindent\\
\textbf{Proof:}
Fix $U$, $\rho_0$, and $k\ne0$ with $a_k>0$, and restrict to $0<\nu\le1$. Define the fixed finite quantities
\[
R_k:=\frac{\|A_kg_k\|_2}{a_k},\qquad
B_k^\ast:=2(R_k+|k|M),
\]
and
\[
B_0^\ast:=\frac{2(\|\Delta\rho_0\|_2+M\|\partial_x\rho_0\|_2)}{N}.
\]
Then
\[
\beta_k=\frac{2(\nu\|A_kg_k\|_2+|k|Ma_k)}{a_k}
\le B_k^\ast,
\qquad
\beta_0\le B_0^\ast.
\]

The defining inequalities for $m_k$ are
\[
C_{\rm res}\frac{k^2M^2}{\nu^2m^2}\le\frac14,
\qquad
C_{\rm res}\frac1{\nu m}\le\frac{\delta_k}{16}=\frac1{16\beta_k}.
\]
Both are guaranteed if
\[
m\ge \frac{2C_{\rm res}^{1/2}|k|M}{\nu},
\qquad
m\ge \frac{16C_{\rm res}\beta_k}{\nu}.
\]
Using $\beta_k\le B_k^\ast$, the least such integer satisfies
\[
m_k\le 1+\frac{2C_{\rm res}^{1/2}|k|M+16C_{\rm res}B_k^\ast}{\nu}
\le C(U,\rho_0,k)\nu^{-1}.
\]
Therefore
\[
\Lambda_k=\nu(k^2+m_k^2+m_k+\tfrac12)
\le C(U,\rho_0,k)\nu^{-1}.
\]

It remains to bound $\delta_k^{-1}\log(\theta_k^{-1})$. Since $\delta_k^{-1}=\beta_k\le B_k^\ast$, this term is zero when $\theta_k=1$. If $\theta_k<1$, then
\[
\log(\theta_k^{-1})
=\frac12\log\left(\frac{32D_k}{\delta_k}\right)
=\frac12\log\left(\frac{32C_{\rm res}\beta_k}{\nu m_k}\right).
\]
Because $m_k\ge1$, $\beta_k\le B_k^\ast$, and $0<\nu\le1$,
\[
\log(\theta_k^{-1})
\le \frac12\log\left(\frac{32C_{\rm res}B_k^\ast}{\nu}\right)_+
\le C(U,\rho_0,k)(1+\nu^{-1}),
\]
where $x_+=\max\{x,0\}$ and we used $\log(1/\nu)\le\nu^{-1}$. Multiplying by $\beta_k\le B_k^\ast$ gives
\[
\delta_k^{-1}\log(\theta_k^{-1})
\le C(U,\rho_0,k)(1+\nu^{-1}).
\]
Combining these bounds gives $\gamma_k\le C(U,\rho_0,k)\nu^{-1}$. Since the $C_k$ of STEP6 also contains only $\beta_0$ and the fixed factor $\delta_0^{-1}\log(N/a_k)=\beta_0\log(N/a_k)$, we get
\[
C_k\le C(U,\rho_0,k)\nu^{-1}.
\]
The optimized $c_2$ in STEP6 is no larger than this $C_k$, proving $c_2\le C(U,\rho_0,k)\nu^{-1}$.

If the datum is $x$-independent and fixed, then $\partial_x\rho_0=0$, so
\[
\beta_0=\frac{2\nu\|\Delta\rho_0\|_2}{N}=O_{\rho_0}(\nu).
\]
Choosing any fixed $l\ne0$ with $g_0^l\ne0$ in STEP6 gives
\[
C_l^{(0)}
=\max\left\{\beta_0,\nu l^2+\beta_0\log(N/|g_0^l|)\right\}
\le C(\rho_0)\nu.
\]
For the whole class, however, $\rho_0$ may depend on $\nu$, and the worst-case power supplied by the explicit formula is $\nu^{-1}$.

\textbf{Dependencies:} STEP3, STEP4, STEP5, STEP6.

\subsubsection*{STEP8: Sharpness construction (KEY STEP)}
\noindent\\
\textbf{Claim:} For each $0<\nu\le1$, set
\[
n_\nu:=\lceil\nu^{-1}\rceil,\qquad U_\nu(t,y):=0,\qquad
\rho_{0,\nu}(x,y):=\cos(n_\nu y).
\]
Then $U_\nu\in L^\infty_{t,y}$, $\rho_{0,\nu}\in C^\infty$, $\rho_{0,\nu}$ is mean-zero and nonzero, and the solution is
\[
\rho_\nu(t,x,y)=e^{-\nu n_\nu^2t}\cos(n_\nu y).
\]
Since
\[
\nu^{-1}\le \nu n_\nu^2\le4\nu^{-1},
\]
for every $t\ge0$,
\[
\|\rho_\nu(t)\|_2
=\|\rho_{0,\nu}\|_2e^{-\nu n_\nu^2t}
\le \|\rho_{0,\nu}\|_2e^{-\nu^{-1}t}.
\]
Hence the sharpness statement holds with
\[
f(\nu)=\nu^{-1},\qquad \widetilde C=1.
\]
\noindent
\textbf{Proof:}
\begin{key-original-step}
For $0<\nu\le1$, $n_\nu=\lceil\nu^{-1}\rceil$ is a positive integer. The function $\cos(n_\nu y)$ is smooth, nonzero, and has zero mean on $\mathbb T$. The choice $U_\nu=0$ belongs to every regularity class listed in the problem, including $L^\infty_{t,y}$.

With $U_\nu=0$, the equation is the heat equation
\[
\partial_t\rho=\nu\Delta\rho.
\]
Since $\rho_{0,\nu}$ is independent of $x$ and
\[
\partial_{yy}\cos(n_\nu y)=-n_\nu^2\cos(n_\nu y),
\]
the unique solution is
\[
\rho_\nu(t,x,y)=e^{-\nu n_\nu^2t}\cos(n_\nu y).
\]
Taking $L^2$ norms gives
\[
\|\rho_\nu(t)\|_2=\|\rho_{0,\nu}\|_2e^{-\nu n_\nu^2t}.
\]

Because $n_\nu=\lceil\nu^{-1}\rceil$,
\[
\nu^{-1}\le n_\nu\le\nu^{-1}+1\le2\nu^{-1}
\]
for $0<\nu\le1$. Multiplying the squared inequality by $\nu$ gives
\[
\nu^{-1}\le \nu n_\nu^2\le4\nu^{-1}.
\]
Therefore $\nu n_\nu^2\ge\nu^{-1}$, and for every $t\ge0$,
\[
e^{-\nu n_\nu^2t}\le e^{-\nu^{-1}t}.
\]
Thus
\[
\|\rho_\nu(t)\|_2
\le\|\rho_{0,\nu}\|_2e^{-\nu^{-1}t}.
\]
\end{key-original-step}
\begin{heuristics}
Diffusion damps a vertical Fourier mode $e^{iny}$ at the exact rate $\nu n^2$. The construction chooses a smooth datum for each fixed $\nu$, but places it at frequency $n\sim\nu^{-1}$. This makes the heat decay rate $\nu n^2$ exactly of order $\nu^{-1}$, so no general statement allowing $\rho_0$ to depend on $\nu$ can improve the negative power $f(\nu)=\nu^{-1}$.
\end{heuristics}

\textbf{Dependencies:} STEP2, STEP7.

\subsubsection*{GOAL: Main Result}
\noindent\\
\textbf{Claim:} The assertion in the problem statement holds.
\noindent\\
\textbf{Proof:}
STEPS 2--6 give an explicit admissible lower-decay exponent $c_2$. Namely, if a nonzero $x$-mode is present, define $c_2$ by
\[
c_2=\min_{\substack{k\ne0\\a_k>0}}
\max\left\{\beta_0,\,
\gamma_k+\delta_0^{-1}\log\left(\frac{N}{a_k}\right)\right\},
\]
with
\[
\beta_0=\frac{2(\nu\|\Delta\rho_0\|_2+M\|\partial_x\rho_0\|_2)}{N},
\quad
\delta_0=\beta_0^{-1},
\]
\[
\gamma_k=
\max\left\{\beta_k,\,
\Lambda_k+\delta_k^{-1}\log(\theta_k^{-1})\right\},
\quad
\beta_k=\frac{2(\nu\|A_kg_k\|_2+|k|Ma_k)}{a_k},
\quad
\delta_k=\beta_k^{-1},
\]
\[
\Lambda_k=\nu(k^2+m_k^2+m_k+\tfrac12),
\quad
\theta_k=\min\left\{1,\left(\frac{\delta_k}{32D_k}\right)^{1/2}\right\},
\quad
D_k=\frac{10^4}{\nu m_k},
\]
where $m_k$ is the least integer $m\ge1$ satisfying
\[
10^4\frac{k^2M^2}{\nu^2m^2}\le\frac14,
\qquad
10^4\frac1{\nu m}\le\frac{\delta_k}{16}.
\]
Then
\[
\|\rho(t)\|_2\ge \|\rho_0\|_2e^{-c_2t}\qquad(t\ge0).
\]

If no nonzero $x$-mode is present, define
\[
c_2=\min_{\substack{l\ne0\\|g_0^l|>0}}
\max\left\{\beta_0,\,
\nu l^2+\delta_0^{-1}\log\left(\frac{N}{|g_0^l|}\right)\right\}.
\]
STEP6 proves the same lower bound in that case. STEP1 gives the upper heat-scale bound
\[
\|\rho(t)\|_2\le\|\rho_0\|_2e^{-\nu t}.
\]
Thus part (a) is proved.
\noindent\\
STEP7 shows that, for fixed $U,\rho_0$ and any fixed admissible nonzero $x$-mode, this explicit $L^\infty$-based formula has worst-case negative-power scaling
\[
f(\nu)=\nu^{-1}.
\]
STEP8 constructs $U_\nu=0$ and $\rho_{0,\nu}=\cos(\lceil\nu^{-1}\rceil y)$ for which
\[
\|\rho_\nu(t)\|_2\le \|\rho_{0,\nu}\|_2e^{-\nu^{-1}t}.
\]
This proves the requested sharpness of the $\nu$-dependence with $\widetilde C=1$, completing part (b).

\textbf{Dependencies:} STEP1, STEP6, STEP7, STEP8.
\end{proof}

\begin{remark}
This proof was produced by QED, supported by a Codex coding agent integrated with ChatGPT 5.5. 
\end{remark}

Next, by the theorem above, we find the explicit estimate of the mixing scale for the advection diffusion equation \eqref{shearnu36}.

\begin{theorem}
Consider the advection-diffusion equation on $\mathbb{T}^2$:
\begin{equation}
    \partial_t \rho + U(t,y)\,\partial_x \rho = \nu \Delta \rho, \qquad \rho(0)=\rho_0 .
\end{equation}
Assume $\rho_0\neq 0$, $\rho_0\in C^\infty$, $\int \rho_0 =0$, $\nu>0$, $U$ is real-valued and
$\|U\|_{L^\infty_{t,y}} < \infty$, $0<\nu\ll1$. Then, we have
\[
\frac{\|\rho\|_{H^{-1}}}{\|\rho\|_{L^2}}\geq c_*.
\] 
For some $c_*=c_*(\rho_0,U,\nu)$ explicitly. 
\end{theorem}
\begin{remark}
\noindent
\begin{enumerate}
\item The authors only edit the wording of this proof, and a few typos. Because of the notation, the AI originally produced the proof for the inhomogeneous $H^{-1}$ norm. Only a few constants change by a factor of $\sqrt{2}$ when converting to the homogeneous norm. To make it clear that the proof is done by AI, the authors keep all the format of AI's proof.
\item One may find that this proof is a bit tedious. The reason this proof is not as clear is the authors ask AI to optimize the estimate, and to use the lower regularity of $U$. To see a more reader-friendly proof with a relatively loose estimate and smooth $U$, one can check \cite{QED2ndp4}, which is a self-contain proof only assuming Theorem \ref{thm:ads}.
\item The last example of $U,\rho_0$ of this theorem shows that the scaling of $c_*$ is not very clear. One may expect a better physical explanation for the future work. 
\item The proof of the existence of such $c_*$ (which can be seen as Batchelor Scale as discussed in \cite{miles2018diffusion}) can get directly from exponential lower bound as in \eqref{eq:L2}, but regardless the explicit expression of the rate or not.
\end{enumerate}
\end{remark}
\begin{proof}
We use the normalized Fourier convention on $\mathbb T^2$, so that Parseval has no extra volume factors. The same ratios are obtained with the unnormalized convention because the common $L^2$ volume factor cancels. Set
\[
M:=\|U\|_{L^\infty_{t,y}},\qquad \mathcal N:=\|\rho_0\|_{L^2}>0.
\]
All norms below are taken on the appropriate torus unless the variables are displayed.

\subsubsection*{STEP1: Fourier decomposition and mode identities}
\noindent\\
\textbf{Claim:} Let
\[
  \rho(t,x,y)=\sum_{k\in\mathbb Z}f_k(t,y)e^{ikx},\qquad
  f_k(t,y)=\sum_{l\in\mathbb Z}\widehat f_k(t,l)e^{ily},
\]
with $g_k^l=\widehat f_k(0,l)$, $g_k(y)=\sum_lg_k^le^{ily}$,
\[
  E_k(t)=\|f_k(t)\|_{L^2_y}^2,\qquad
  a_k=E_k(0)^{1/2},\qquad
  \mathcal N=\|\rho_0\|_2=\left(\sum_k a_k^2\right)^{1/2},
\]
and
\[
  F_k(t)=
  \begin{cases}
  \displaystyle\sum_{l\in\mathbb Z}\frac{|\widehat f_k(t,l)|^2}{k^2+l^2},&k\ne0,\\[1ex]
  \displaystyle\sum_{l\ne0}\frac{|\widehat f_0(t,l)|^2}{l^2},&k=0.
  \end{cases}
\]
Then
\[
  \|\rho(t)\|_2^2=\sum_kE_k(t),\qquad
  \|\rho(t)\|_{\dot H^{-1}}^2=\sum_kF_k(t),
\]
and each mode satisfies
\[
  \partial_t f_k+\nu(k^2-\partial_{yy})f_k=-ikU(t,y)f_k,
\]
\[
  \frac12E_k'(t)=-\nu\left(k^2E_k(t)+\|\partial_yf_k(t)\|_2^2\right).
\]
If $P_{\le N}$ is projection to $|l|\le N$, $L_{k,N}(t)=\|P_{\le N}f_k(t)\|_2^2$, and
$H_{k,N}(t)=E_k(t)-L_{k,N}(t)$, then the low-high transfer term satisfies
\[
  |T_{k,N}(t)|\le |k|\,M\,L_{k,N}(t)^{1/2}H_{k,N}(t)^{1/2},
  \qquad M=\|U\|_{L^\infty_{t,y}}.
\]
\noindent
\textbf{Proof:}
On every finite interval $[0,T]$, the drift $b(t,x,y)=(U(t,y),0)$ is divergence-free and belongs to
$L^1([0,T];L^2(\mathbb T^2))\cap L^2([0,T];L^2(\mathbb T^2))$.
Thus the parabolic solution class is available and unique for bounded shear drifts. See \cite{bonicatto2024weak}. 

The Fourier identities below are obtained first for smooth $U$. For merely bounded $U$, take smooth $U^n\to U$ in $L^2([0,T]\times\mathbb T)$ with $\|U^n\|_\infty\le M$, pass to the parabolic limit using uniqueness, and use the boundedness of the finite Fourier projections on $L^2_y$ to pass the projected identities and inequalities to the limit on $[0,T]$. Since $T$ is arbitrary, the identities hold on every finite time interval.

Since $U$ is independent of $x$, multiplication by $U$ does not couple different $x$-Fourier modes. Taking the coefficient of $e^{ikx}$ in $\partial_t\rho+U\partial_x\rho=\nu(\partial_{xx}+\partial_{yy})\rho$ gives
\[
\partial_t f_k+ikUf_k=\nu(\partial_{yy}-k^2)f_k,
\]
which is the stated mode equation.

Parseval gives
\[
\|\rho(t)\|_2^2=\sum_{k,l}|\widehat f_k(t,l)|^2=\sum_kE_k(t).
\]
The full mean is conserved: integrating the equation over $\mathbb T^2$ gives
\[
\frac{d}{dt}\int_{\mathbb T^2}\rho(t,x,y)\,dx\,dy=0,
\]
because $\int U(t,y)\partial_x\rho\,dxdy=0$ by periodicity in $x$, and $\int\Delta\rho\,dxdy=0$. Since $\int\rho_0=0$, the $(0,0)$ coefficient is absent for every $t\ge0$. Therefore the homogeneous negative Sobolev norm is
\[
\|\rho(t)\|_{\dot H^{-1}}^2
=\sum_{(k,l)\ne(0,0)}\frac{|\widehat f_k(t,l)|^2}{k^2+l^2}
=\sum_kF_k(t).
\]

Taking the real part of the $L^2_y$ inner product of the mode equation with $f_k$ gives
\[
\frac12\frac{d}{dt}\|f_k\|_2^2+\nu k^2\|f_k\|_2^2+\nu\|\partial_yf_k\|_2^2
=\operatorname{Re}\int_{\mathbb T}(-ikU)|f_k|^2\,dy.
\]
The last term is zero because $U$ is real-valued and $-ik\int U|f_k|^2$ is purely imaginary. This proves the mode energy identity.

Now write $f^L=P_{\le N}f_k$ and $f^H=P_{>N}f_k$. The transfer term appearing in the low/high energy identities is
\[
T_{k,N}(t):=\operatorname{Re}\langle -ik\,U(t,\cdot)f^L(t),f^H(t)\rangle_{L^2_y}.
\]
By Cauchy's inequality and the $L^\infty$ bound on $U$,
\[
|T_{k,N}(t)|
\le |k|\,M\,\|f^L(t)\|_2\|f^H(t)\|_2
=|k|\,M\,L_{k,N}(t)^{1/2}H_{k,N}(t)^{1/2}.
\]

\textbf{Dependencies:} \cite{bonicatto2024weak}.

\subsubsection*{STEP2: The $L^2$ lower exponent and mode upper bounds}
\noindent\\
\textbf{Claim:} Let $c_2$ be the explicit admissible exponent from Theorem \ref{thm:ads}. Then
\[
  \|\rho(t)\|_2\ge \mathcal N e^{-c_2t}\qquad(t\ge0).
\]
In addition, for every $k\in\mathbb Z$,
\[
  E_k(t)\le a_k^2e^{-2\nu k^2t}\qquad(t\ge0).
\]
\noindent
\textbf{Proof:}
For explicitness, the exponent $c_2$ used below is the following one from Theorem \ref{thm:ads}. For $k\ne0$ with $a_k>0$, set
\[
A_k=k^2-\partial_{yy},\qquad
L_k=\nu\|A_kg_k\|_2+|k|Ma_k,\qquad
\beta_k=\frac{2L_k}{a_k},\qquad \delta_k=\beta_k^{-1}.
\]
Let $C_{\rm res}=10^4$. Let $m_k$ be the least integer $m\ge1$ satisfying
\[
C_{\rm res}\frac{k^2M^2}{\nu^2m^2}\le\frac14,\qquad
C_{\rm res}\frac1{\nu m}\le\frac{\delta_k}{16}.
\]
Define
\[
\Lambda_k=\nu\left(k^2+m_k^2+m_k+\frac12\right),\qquad
D_k=\frac{C_{\rm res}}{\nu m_k},\qquad
\theta_k=\min\left\{1,\left(\frac{\delta_k}{32D_k}\right)^{1/2}\right\},
\]
\[
\gamma_k=\max\left\{\beta_k,\,
\Lambda_k+\delta_k^{-1}\log(\theta_k^{-1})\right\}.
\]
Also set
\[
L_0=\nu\|\Delta\rho_0\|_2+M\|\partial_x\rho_0\|_2,\qquad
\beta_0=\frac{2L_0}{\mathcal N},\qquad \delta_0=\beta_0^{-1}.
\]
If some nonzero $x$-mode is present, define
\[
c_2=
\min_{\substack{k\ne0\\a_k>0}}
\max\left\{\beta_0,\gamma_k+\delta_0^{-1}\log\left(\frac{\mathcal N}{a_k}\right)\right\}.
\]
If $a_k=0$ for all $k\ne0$, define instead
\[
c_2=
\min_{\substack{l\ne0\\|g_0^l|>0}}
\max\left\{\beta_0,\nu l^2+\delta_0^{-1}\log\left(\frac{\mathcal N}{|g_0^l|}\right)\right\}.
\]
The local proof establishes exactly that this $c_2$ is admissible:
\[
\|\rho(t)\|_2\ge\mathcal N e^{-c_2t}\qquad(t\ge0).
\]

For the mode upper bound, the energy identity from STEP1 gives
\[
\frac12E_k'(t)
=-\nu\left(k^2E_k(t)+\|\partial_yf_k(t)\|_2^2\right)
\le-\nu k^2E_k(t).
\]
Thus $E_k'(t)\le-2\nu k^2E_k(t)$. Gronwall's inequality yields
\[
E_k(t)\le E_k(0)e^{-2\nu k^2t}=a_k^2e^{-2\nu k^2t}.
\]

\textbf{Dependencies:} Theorem \ref{thm:ads}, STEP1.

\subsubsection*{STEP3: All-time finite $x$-frequency retention}

\textbf{Claim:} Define
\[
  K_c=\left\lceil\left(\frac{2c_2}{\nu}\right)^{1/2}\right\rceil,\qquad
  K_0=\min\left\{J\in\mathbb N_0:\ \sum_{|k|>J}a_k^2\le\frac{\mathcal N^2}{2}\right\},
\]
and set
\[
  K=\max\{K_c,K_0\},\qquad
  \mathcal K=\{k\in\mathbb Z:\ |k|\le K,\ a_k>0\}.
\]
Then $K<\infty$, $\mathcal K\ne\varnothing$, and for every $t\ge0$,
\[
  \sum_{|k|>K}E_k(t)\le \frac12\|\rho(t)\|_2^2,
  \qquad
  \sum_{k\in\mathcal K}E_k(t)\ge \frac12\|\rho(t)\|_2^2.
\]

\textbf{Proof:}
Since $\sum_ka_k^2=\mathcal N^2<\infty$, the tails $\sum_{|k|>J}a_k^2$ tend to zero as $J\to\infty$. Hence $K_0<\infty$, and $K_c<\infty$ because $c_2<\infty$ and $\nu>0$. Thus $K<\infty$.

For $|k|>K$, one has $|k|\ge K+1$. By STEP2,
\[
\sum_{|k|>K}E_k(t)
\le e^{-2\nu(K+1)^2t}\sum_{|k|>K}a_k^2.
\]
Because $K\ge K_0$,
\[
\sum_{|k|>K}a_k^2\le \sum_{|k|>K_0}a_k^2\le\frac{\mathcal N^2}{2}.
\]
Because $K\ge K_c$,
\[
\nu(K+1)^2>c_2.
\]
Therefore, using the $L^2$ lower bound from STEP2,
\[
\sum_{|k|>K}E_k(t)
\le\frac{\mathcal N^2}{2}e^{-2\nu(K+1)^2t}
\le\frac{\mathcal N^2}{2}e^{-2c_2t}
\le \frac12\|\rho(t)\|_2^2.
\]

If $\mathcal K$ were empty, then $a_k=0$ for every $|k|\le K$, so at $t=0$
\[
\mathcal N^2=\sum_{|k|>K}a_k^2\le\frac{\mathcal N^2}{2},
\]
contradicting $\mathcal N>0$. Thus $\mathcal K\ne\varnothing$. Finally, modes with $a_k=0$ have zero initial data; by uniqueness for the linear mode equation, $E_k(t)=0$ for all $t\ge0$. Hence
\[
\sum_{k\in\mathcal K}E_k(t)=\sum_{|k|\le K}E_k(t)
=\|\rho(t)\|_2^2-\sum_{|k|>K}E_k(t)
\ge\frac12\|\rho(t)\|_2^2.
\]

\textbf{Dependencies:} STEP2.

\subsubsection*{STEP4: All-time vertical finite-window retention (KEY STEP)}
\noindent\\
\textbf{Claim:} For every $k\in\mathcal K$, define an initial vertical half-energy cutoff
\[
  J_k=\min\left\{J\in\mathbb N:\ \sum_{|l|>J}|g_k^l|^2\le\frac{a_k^2}{2}\right\}.
\]
For $k=0$, put $N_0=J_0$. For $k\ne0$, put
\[
  N_k=\max\left\{J_k,\left\lceil\frac{|k|M}{\nu}\right\rceil,1\right\}.
\]
Then $J_k,N_k<\infty$, and for every $t\ge0$:
\[
  L_{0,N_0}(t)\ge\frac12E_0(t),\qquad
  F_0(t)\ge\frac{1}{2N_0^2}E_0(t)
\]
whenever $0\in\mathcal K$, and
\[
  L_{k,N_k}(t)\ge\frac12E_k(t),\qquad
  F_k(t)\ge\frac{1}{2(k^2+N_k^2)}E_k(t)
\]
for every $k\in\mathcal K$ with $k\ne0$.
\noindent\\
\textbf{Proof:}
\begin{key-original-step}
Fix $k\in\mathcal K$. Since $g_k\in L^2(\mathbb T_y)$, its Fourier tail tends to zero. Because $a_k>0$, there exists $J\in\mathbb N$ for which
\[
\sum_{|l|>J}|g_k^l|^2\le\frac{a_k^2}{2}.
\]
Thus $J_k<\infty$, and $N_k<\infty$ follows immediately from the definition.

First consider $k=0$. Since the full spatial mean of $\rho_0$ is zero, $g_0^0=0$, and the $k=0$ mode solves the one-dimensional heat equation
\[
\partial_t f_0=\nu\partial_{yy}f_0,\qquad
\widehat f_0(t,l)=e^{-\nu l^2t}g_0^l.
\]
Set $N=N_0=J_0$. Then
\[
H_{0,N}(t)=\sum_{|l|>N}|g_0^l|^2e^{-2\nu l^2t}
\le e^{-2\nu(N+1)^2t}\sum_{|l|>N}|g_0^l|^2,
\]
while, because $g_0^0=0$,
\[
L_{0,N}(t)=\sum_{0<|l|\le N}|g_0^l|^2e^{-2\nu l^2t}
\ge e^{-2\nu N^2t}\sum_{0<|l|\le N}|g_0^l|^2.
\]
The definition of $J_0=N$ gives $H_{0,N}(0)\le a_0^2/2$, hence
\[
L_{0,N}(0)=a_0^2-H_{0,N}(0)\ge a_0^2/2\ge H_{0,N}(0).
\]
Therefore
\[
H_{0,N}(t)\le e^{-2\nu(N+1)^2t}H_{0,N}(0)
\le e^{-2\nu N^2t}L_{0,N}(0)
\le L_{0,N}(t).
\]
Since $E_0=L_{0,N}+H_{0,N}$, this proves
\[
L_{0,N_0}(t)\ge\frac12E_0(t).
\]
Moreover,
\[
F_0(t)
=\sum_{l\ne0}\frac{|\widehat f_0(t,l)|^2}{l^2}
\ge \sum_{0<|l|\le N_0}\frac{|\widehat f_0(t,l)|^2}{l^2}
\ge \frac{1}{N_0^2}L_{0,N_0}(t)
\ge \frac{1}{2N_0^2}E_0(t).
\]

Now fix $k\in\mathcal K$ with $k\ne0$, and write $N=N_k$, $L=L_{k,N}$, $H=H_{k,N}$. At $t=0$,
\[
H(0)=\sum_{|l|>N}|g_k^l|^2
\le\sum_{|l|>J_k}|g_k^l|^2
\le\frac{a_k^2}{2},
\]
so $L(0)=a_k^2-H(0)\ge H(0)$.

Let $f^L=P_{\le N}f_k$, $f^H=P_{>N}f_k$, and
\[
Z(t):=H(t)-L(t).
\]
The low/high energy identities obtained by projecting the mode equation are
\[
\frac12H'(t)+\nu\langle (k^2-\partial_{yy})f^H(t),f^H(t)\rangle=T(t),
\]
\[
\frac12L'(t)+\nu\langle (k^2-\partial_{yy})f^L(t),f^L(t)\rangle=-T(t),
\]
where
\[
T(t)=\operatorname{Re}\langle -ikU(t,\cdot)f^L(t),f^H(t)\rangle.
\]
These identities hold for a.e. $t$, and $L,H,Z$ are absolutely continuous. Subtracting the low identity from the high identity gives
\[
Z'(t)=-2\nu\Bigl(
\langle (k^2-\partial_{yy})f^H,f^H\rangle
-\langle (k^2-\partial_{yy})f^L,f^L\rangle
\Bigr)+4T(t)
\]
for a.e. $t$.

On the set where $Z(t)\ge0$, one has $L(t)\le H(t)$. The high block contains only $|l|\ge N+1$, while the low block contains only $|l|\le N$. Hence
\[
\langle (k^2-\partial_{yy})f^H,f^H\rangle
\ge (k^2+(N+1)^2)H,
\]
\[
\langle (k^2-\partial_{yy})f^L,f^L\rangle
\le (k^2+N^2)L.
\]
Therefore
\[
\begin{aligned}
&\langle (k^2-\partial_{yy})f^H,f^H\rangle
-\langle (k^2-\partial_{yy})f^L,f^L\rangle\\
&\qquad\ge (k^2+N^2)(H-L)+(2N+1)H
=(k^2+N^2)Z+(2N+1)H.
\end{aligned}
\]
Also, by STEP1,
\[
|T(t)|\le |k|M\sqrt{L(t)H(t)}\le |k|M\,H(t)
\]
on the same set $Z(t)\ge0$. Combining the last two estimates gives, for a.e. $t$ such that $Z(t)\ge0$,
\[
Z'(t)\le
-2\nu(k^2+N^2)Z(t)+\bigl[-2\nu(2N+1)+4|k|M\bigr]H(t).
\]
Since $N=N_k\ge |k|M/\nu$,
\[
-2\nu(2N+1)+4|k|M
=-4\nu N-2\nu+4|k|M
\le -2\nu.
\]
Thus, for a.e. $t$ with $Z(t)\ge0$,
\[
Z'(t)\le -2\nu(k^2+N^2)Z(t)-2\nu H(t)\le0.
\]

Let $Z_+(t)=\max\{Z(t),0\}$. Since $Z$ is absolutely continuous,
\[
\frac{d}{dt}Z_+(t)=\mathbf 1_{\{Z(t)>0\}}Z'(t)
\]
for a.e. $t$. The preceding inequality implies $(Z_+)'(t)\le0$ for a.e. $t$. As $Z_+(0)=0$, it follows that $Z_+(t)=0$ for every $t\ge0$. Hence $H(t)\le L(t)$ for all $t\ge0$, and therefore
\[
L_{k,N_k}(t)\ge\frac12E_k(t).
\]
Finally,
\[
F_k(t)
\ge \sum_{|l|\le N_k}\frac{|\widehat f_k(t,l)|^2}{k^2+l^2}
\ge\frac{1}{k^2+N_k^2}L_{k,N_k}(t)
\ge\frac{1}{2(k^2+N_k^2)}E_k(t).
\]
\end{key-original-step}
\begin{heuristics}
The low/high boundary is stable because diffusion sees a spectral gap $2N+1$ across the vertical cutoff. If the high block is at least as large as the low block, then the high block dissipates faster by at least $\nu(2N+1)H$, while the shear can transfer energy across the cutoff at rate no more than $|k|M\sqrt{LH}\le |k|MH$. Choosing $N\ge |k|M/\nu$ makes this net flux point from high dominance back toward low dominance. The positive-part argument turns this boundary computation into a genuine all-time invariant.
\end{heuristics}

\textbf{Dependencies:} STEP1, STEP3.

\subsubsection*{STEP5: The explicit all-time $H^{-1}/L^2$ lower bound}
\noindent\\
\textbf{Claim:} Define
\[
  R_*^2=
  \max\left(
  \{k^2+N_k^2:\ k\in\mathcal K,\ k\ne0\}
  \cup
  \{N_0^2:\ 0\in\mathcal K\}
  \right),
\]
with the $N_0$-term omitted if $0\notin\mathcal K$. Then $0<R_*<\infty$ and for every
$t\ge0$,
\[
  \frac{\|\rho(t)\|_{\dot H^{-1}}}{\|\rho(t)\|_{L^2}}\ge \frac{1}{2R_*}.
\]
Consequently, for the homogeneous convention
\[
  \|h\|_{H^{-1}}^2=\sum_{(k,l)\ne(0,0)}\frac{|\widehat h(k,l)|^2}{k^2+l^2},
\]
one has for every $t\ge0$
\[
  \frac{\|\rho(t)\|_{H^{-1}}}{\|\rho(t)\|_{L^2}}
  \ge c_*:=\frac{1}{2\,R_*}.
\]
\noindent
\textbf{Proof:}
The set $\mathcal K$ is finite and nonempty by STEP3. For $k\in\mathcal K$, the corresponding quantity $N_0^2$ or $k^2+N_k^2$ is a positive finite number: $N_0=J_0\ge1$, and for $k\ne0$ the term $k^2+N_k^2\ge1$. Hence $0<R_*<\infty$.

By STEP4, for every $k\in\mathcal K$,
\[
F_k(t)\ge \frac{1}{2R_*^2}E_k(t).
\]
Summing over $k\in\mathcal K$ and then using STEP3 gives
\[
\|\rho(t)\|_{\dot H^{-1}}^2
=\sum_kF_k(t)
\ge\sum_{k\in\mathcal K}F_k(t)
\ge\frac{1}{2R_*^2}\sum_{k\in\mathcal K}E_k(t)
\ge\frac{1}{4R_*^2}\|\rho(t)\|_2^2.
\]
Taking square roots proves
\[
\frac{\|\rho(t)\|_{\dot H^{-1}}}{\|\rho(t)\|_2}\ge\frac1{2R_*}.
\]

\textbf{Dependencies:} STEP3, STEP4.

\subsubsection*{STEP6: Explicit $\nu$-dependence of the admissible exponent (KEY STEP)}
\noindent\\
\textbf{Claim:} Let
\[
  J_*=\max_{k\in\mathcal K}J_k.
\]
Then
\[
  K=\max\left\{
  \left\lceil\left(\frac{2c_2}{\nu}\right)^{1/2}\right\rceil,
  \min\left\{J\in\mathbb N_0:\ \sum_{|k|>J}a_k^2\le\frac{\mathcal N^2}{2}\right\}
  \right\},
\]
and
\[
  R_*\le 1+K+J_*+\frac{KM}{\nu}.
\]
Hence the explicit all-time admissible exponent satisfies
\[
  c_*=\frac{1}{2\,R_*}
  \ge
  \frac{1}{2\left(1+K+J_*+KM/\nu\right)}.
\]
More explicitly, for $0<\nu\le1$, if
\[
  C_K:=1+\sqrt{2C_2}+C_x,\qquad
  M=0,\qquad c_2\le C_2\nu^{-1},\qquad
  K_0\le C_x\nu^{-1},\qquad J_*\le C_y\nu^{-1},
\]
then
\[
  K\le C_K\nu^{-1},\qquad
  c_*\ge \frac{\nu}{2(1+C_K+C_y)}.
\]
If instead
\[
  M\le C_M,\qquad c_2\le C_2\nu^{-1},\qquad
  K_0\le C_x\nu^{-1},\qquad J_*\le C_y\nu^{-1},
\]
with the same $C_K$, then
\[
  K\le C_K\nu^{-1},\qquad
  c_*\ge \frac{\nu^2}{2(1+C_K+C_y+C_KC_M)}.
\]
\noindent
\textbf{Proof:}
\begin{key-original-step}
The displayed formula for $K$ is exactly the definition from STEP3, with $K_0$ written out. Since $\mathcal K$ is finite and $J_k<\infty$ for every $k\in\mathcal K$, $J_*<\infty$.

Fix $k\in\mathcal K$ with $k\ne0$. By definition,
\[
N_k=\max\left\{J_k,\left\lceil\frac{|k|M}{\nu}\right\rceil,1\right\}.
\]
Since $J_k\le J_*$ and $|k|\le K$,
\[
\left\lceil\frac{|k|M}{\nu}\right\rceil\le \frac{KM}{\nu}+1.
\]
Thus
\[
N_k\le J_*+\frac{KM}{\nu}+1.
\]
Consequently
\[
\sqrt{k^2+N_k^2}\le |k|+N_k
\le K+J_*+\frac{KM}{\nu}+1.
\]
If $0\in\mathcal K$, then
\[
N_0=J_0\le J_*\le 1+K+J_*+\frac{KM}{\nu}.
\]
Taking the maximum over the finitely many terms in the definition of $R_*$ gives
\[
R_*\le 1+K+J_*+\frac{KM}{\nu}.
\]
The lower bound for $c_*$ follows immediately from STEP5:
\[
c_*=\frac1{2R_*}
\ge
\frac{1}{2\left(1+K+J_*+KM/\nu\right)}.
\]

Now assume $0<\nu\le1$ and $c_2\le C_2\nu^{-1}$, $K_0\le C_x\nu^{-1}$. Then
\[
K_c=\left\lceil\left(\frac{2c_2}{\nu}\right)^{1/2}\right\rceil
\le \sqrt{2C_2}\,\nu^{-1}+1
\le (1+\sqrt{2C_2})\nu^{-1}.
\]
Therefore
\[
K=\max\{K_c,K_0\}
\le (1+\sqrt{2C_2}+C_x)\nu^{-1}
=C_K\nu^{-1}.
\]

If $M=0$ and $J_*\le C_y\nu^{-1}$, then
\[
1+K+J_*
\le 1+(C_K+C_y)\nu^{-1}
\le(1+C_K+C_y)\nu^{-1}.
\]
Substitution into the preceding estimate for $c_*$ gives
\[
c_*\ge \frac{\nu}{2(1+C_K+C_y)}.
\]

If $M\le C_M$ and $J_*\le C_y\nu^{-1}$, then
\[
1+K+J_*+\frac{KM}{\nu}
\le
1+C_K\nu^{-1}+C_y\nu^{-1}+C_KC_M\nu^{-2}.
\]
Since $0<\nu\le1$, both $1$ and $\nu^{-1}$ are bounded above by $\nu^{-2}$. Hence
\[
1+K+J_*+\frac{KM}{\nu}
\le (1+C_K+C_y+C_KC_M)\nu^{-2},
\]
and therefore
\[
c_*\ge \frac{\nu^2}{2(1+C_K+C_y+C_KC_M)}.
\]
\end{key-original-step}
\begin{heuristics}
The exact bound is $R_*^{-1}$, where $R_*$ is the largest resolved Fourier radius needed to keep at least half of the relevant energy. The estimate $R_*\le 1+K+J_*+KM/\nu$ separates the three costs: retaining enough $x$-modes, retaining the initial $y$-frequencies, and adding a shear-transfer barrier of size $|k|M/\nu$. The $\nu$-power corollaries are only consequences of additional assumptions on these three quantities.
\end{heuristics}

\textbf{Dependencies:} STEP3, STEP4, STEP5.

\subsubsection*{STEP7: Heat eigenfunction sharpness construction (KEY STEP)}
\noindent\\
\textbf{Claim:} For every $p>0$ and $0<\nu\le1$, set
\[
  n_\nu=\lceil\nu^{-p}\rceil,\qquad
  U_\nu(t,y)=0,\qquad
  \rho_{0,\nu}(x,y)=\cos(n_\nu y).
\]
Then $\rho_{0,\nu}\in C^\infty$, $\int_{\mathbb T^2}\rho_{0,\nu}=0$, and
\[
  \rho_\nu(t,x,y)=e^{-\nu n_\nu^2t}\cos(n_\nu y).
\]
For every $t\ge0$,
\[
  \frac{\|\rho_\nu(t)\|_{ H^{-1}}}{\|\rho_\nu(t)\|_{L^2}}=\frac1{n_\nu}\le \nu^p,
\]

For this same family, the cutoffs from STEP3--STEP5 satisfy
\[
  \mathcal K_\nu=\{0\},\qquad
  J_{*,\nu}=J_{0,\nu}=N_{0,\nu}=R_{*,\nu}=n_\nu,\qquad
  c_{*,\nu}=\frac1{2\,n_\nu}.
\]
Taking $p=1$ gives
\[
  M=0,\qquad K_{0,\nu}=0,\qquad J_{*,\nu}\le2\nu^{-1},\qquad c_{2,\nu}\le C\nu^{-1},
\]
with $C>0$ independent of $\nu$, while
\[
  \frac{\|\rho_\nu(t)\|_{H^{-1}}}{\|\rho_\nu(t)\|_{L^2}}\le \nu
  \qquad(t\ge0),
\]
and for every function $g:(0,1]\to(0,\infty)$ satisfying
\[
  \lim_{\nu\downarrow0}\frac{g(\nu)}{\nu}=+\infty,
\]
there exists $\nu_g>0$ such that for all $0<\nu<\nu_g$,
\[
  \frac{\|\rho_\nu(t)\|_{H^{-1}}}{\|\rho_\nu(t)\|_{L^2}}<g(\nu)
  \qquad(t\ge0).
\]
More generally, for every $p>0$ and every $g_p:(0,1]\to(0,\infty)$ satisfying
\[
  \lim_{\nu\downarrow0}\frac{g_p(\nu)}{\nu^p}=+\infty,
\]
the same family satisfies
\[
  \nu^{-p}\le R_{*,\nu}\le2\nu^{-p},\qquad
  \frac{\nu^p}{4}\le c_{*,\nu}\le\frac{\nu^p}{2},
\]
and
\[
  \frac{\|\rho_\nu(t)\|_{H^{-1}}}{\|\rho_\nu(t)\|_{L^2}}<g_p(\nu)
  \qquad(t\ge0)
\]
for all sufficiently small $\nu$.
\noindent\\
\textbf{Proof:}
\begin{key-original-step}
The datum $\rho_{0,\nu}(x,y)=\cos(n_\nu y)$ is smooth and has zero spatial mean. Since $U_\nu=0$, the equation is the heat equation, and
\[
\Delta\cos(n_\nu y)=-n_\nu^2\cos(n_\nu y).
\]
Thus
\[
\rho_\nu(t,x,y)=e^{-\nu n_\nu^2t}\cos(n_\nu y).
\]

The only nonzero Fourier modes are $(k,l)=(0,n_\nu)$ and $(0,-n_\nu)$. The common heat factor cancels from the ratio of any homogeneous norm to the $L^2$ norm. Therefore
\[
\|\rho_\nu(t)\|_{\dot H^{-1}}^2
=\frac1{n_\nu^2}\|\rho_\nu(t)\|_2^2.
\]
Taking square roots gives
\[
\frac{\|\rho_\nu(t)\|_{H^{-1}}}{\|\rho_\nu(t)\|_2}
=\frac1{n_\nu}.
\]
Since $n_\nu=\lceil\nu^{-p}\rceil\ge\nu^{-p}$, both ratios are bounded above by $\nu^p$.

For the cutoffs, only the $x$-mode $k=0$ has nonzero initial energy. Hence
\[
\mathcal K_\nu=\{0\}.
\]
The initial $y$-energy is supported exactly at $|l|=n_\nu$. If $J<n_\nu$, then
\[
\sum_{|l|>J}|g_0^l|^2=a_0^2>\frac{a_0^2}{2},
\]
whereas if $J\ge n_\nu$, that tail is zero. Thus
\[
J_{0,\nu}=n_\nu.
\]
Since $N_0=J_0$, and $R_*^2=N_0^2$ for the singleton set $\mathcal K_\nu=\{0\}$,
\[
J_{*,\nu}=J_{0,\nu}=N_{0,\nu}=R_{*,\nu}=n_\nu.
\]
The STEP5 constant is therefore
\[
c_{*,\nu}=\frac1{2\,R_{*,\nu}}=\frac1{2\,n_\nu}.
\]

For $0<\nu\le1$, $\nu^{-p}\ge1$, so
\[
\nu^{-p}\le n_\nu=\lceil\nu^{-p}\rceil\le \nu^{-p}+1\le2\nu^{-p}.
\]
Consequently
\[
\nu^{-p}\le R_{*,\nu}\le2\nu^{-p}
\]
and
\[
\frac{\nu^p}{4}\le
\frac1{2\,R_{*,\nu}}
\le\frac{\nu^p}{2}.
\]

Now take $p=1$. Then $M=0$, $K_{0,\nu}=0$, and
\[
J_{*,\nu}=n_\nu=\lceil\nu^{-1}\rceil\le2\nu^{-1}.
\]
The exact heat decay is
\[
\|\rho_\nu(t)\|_2=\|\rho_{0,\nu}\|_2e^{-\nu n_\nu^2t},
\]
so $c=\nu n_\nu^2$ is an admissible $L^2$ lower-decay exponent. Also,
\[
\nu n_\nu^2\le \nu(2\nu^{-1})^2=4\nu^{-1}.
\]
The explicit $c_{2,\nu}$ from STEP2 is likewise $O(\nu^{-1})$: in the $x$-independent branch, its only active vertical frequencies are $l=\pm n_\nu$, the ratio $\mathcal N/|g_0^{\pm n_\nu}|$ is independent of $\nu$, and $\beta_0=2\nu n_\nu^2\le8\nu^{-1}$. Thus $c_{2,\nu}\le C\nu^{-1}$ for a constant $C$ independent of $\nu$.

Finally, if $g_p(\nu)/\nu^p\to+\infty$, then there is $\nu_{g_p}>0$ such that $g_p(\nu)>\nu^p$ whenever $0<\nu<\nu_{g_p}$. For those $\nu$,
\[
\frac{\|\rho_\nu(t)\|_{H^{-1}}}{\|\rho_\nu(t)\|_2}
\le\nu^p<g_p(\nu)
\]
for every $t\ge0$. The stated $p=1$ assertion is the special case $g_p=g$.
\end{key-original-step}
\begin{heuristics}
This construction eliminates shear and isolates the unavoidable frequency dependence. A single heat eigenfunction never changes its normalized negative Sobolev length scale: diffusion multiplies $L^2$ and $H^{-1}$ by the same exponential factor. Thus placing the initial datum at frequency $n_\nu$ forces the ratio to be exactly $n_\nu^{-1}$ in the homogeneous norm and $(1+n_\nu^2)^{-1/2}$ in the inhomogeneous norm. Since the proof's $R_*$ is exactly $n_\nu$ for this family, the lower-bound scale $R_*^{-1}$ cannot be improved beyond constants.
\end{heuristics}

\textbf{Dependencies:} STEP1, STEP6.

\subsubsection*{STEP8: Assembly of the explicit bound and sharpness statement}
\noindent\\
\textbf{Claim:} For every admissible $(\rho_0,U,\nu)$, with $K,J_k,N_k,R_*$ as defined in STEP3--STEP5,
\[
  \inf_{t\ge0}\frac{\|\rho(t)\|_{H^{-1}}}{\|\rho(t)\|_{L^2}}
  \ge
  \frac1{2\,R_*}
  \ge
  \frac{1}{2\left(1+K+J_*+KM/\nu\right)}.
\]
For every $p>0$, the STEP7 family satisfies, for $0<\nu\le1$,
\[
  \nu^{-p}\le R_{*,\nu}=n_\nu\le2\nu^{-p},
\]
hence
\[
  \frac{\nu^p}{4}\le c_{*,\nu}=\frac1{2R_{*,\nu}}\le\frac{\nu^p}{2},
\]
while for every $t\ge0$,
\[
  \frac{\|\rho_\nu(t)\|_{H^{-1}}}{\|\rho_\nu(t)\|_{L^2}}\le\nu^p.
\]
For every function $g_p:(0,1]\to(0,\infty)$ with
\[
  \lim_{\nu\downarrow0}\frac{g_p(\nu)}{\nu^p}=+\infty,
\]
this implies
\[
  \frac{\|\rho_\nu(t)\|_{H^{-1}}}{\|\rho_\nu(t)\|_{L^2}}<g_p(\nu)
  \qquad(t\ge0)
\]
for all sufficiently small $\nu$. In the special case $p=1$, the same family satisfies
\[
  M=0,\qquad K_{0,\nu}=0,\qquad J_{*,\nu}\le2\nu^{-1},\qquad c_{2,\nu}\le C\nu^{-1},
\]
and
\[
  \frac{\|\rho_\nu(t)\|_{H^{-1}}}{\|\rho_\nu(t)\|_{L^2}}\le\nu
  \qquad(t\ge0).
\]
\noindent
\textbf{Proof:}
The first inequality is STEP5, and the second inequality is STEP6. The estimates for the heat eigenfunction family, including the bounds on $R_{*,\nu}$, $c_{*,\nu}$, and the actual $H^{-1}/L^2$ ratio, are exactly STEP7. Thus the explicit lower-bound scale produced by the proof is $R_*^{-1}$, and the family $U_\nu=0$, $\rho_{0,\nu}=\cos(\lceil\nu^{-p}\rceil y)$ shows that this $R_*^{-1}$ dependence is sharp (the lower bound cannot be improved beyond universal constants).

\textbf{Dependencies:} STEP5, STEP6, STEP7.

\subsubsection*{GOAL: Main Result}
\noindent\\
\textbf{Claim:} The original problem holds with the explicit admissible exponent
\[
c_*=\frac1{2\,R_*},
\]
where $R_*$ is constructed from $\rho_0,U,\nu$ by STEPS 2--5. More explicitly,
\[
c_*\ge
\frac{1}{2\left(1+K+J_*+KM/\nu\right)}.
\]
The exact frequency-dependent $\nu$-scale $R_*^{-1}$ is sharp.
\noindent\\
\textbf{Proof:}
Given $\rho_0,U,\nu$, first compute $c_2$ from STEP2. Then define $K$ and $\mathcal K$ by STEP3, $J_k,N_k$ by STEP4, $R_*$ by STEP5, and $J_*=\max_{k\in\mathcal K}J_k$. STEP5 proves, for every $t\ge0$,
\[
\frac{\|\rho(t)\|_{H^{-1}}}{\|\rho(t)\|_{L^2}}
\ge\frac1{2\,R_*}.
\]
STEP6 gives the displayed more explicit lower estimate in terms of $K,J_*,M,\nu$.

For sharpness, STEP7 constructs, for every $p>0$, the admissible family
\[
U_\nu=0,\qquad \rho_{0,\nu}(x,y)=\cos(\lceil\nu^{-p}\rceil y),
\]
for which $R_{*,\nu}=\lceil\nu^{-p}\rceil$, so
\[
c_{*,\nu}\asymp \nu^p,
\]
while
\[
\frac{\|\rho_\nu(t)\|_{H^{-1}}}{\|\rho_\nu(t)\|_{L^2}}
=\frac1{\lceil\nu^{-p}\rceil}
\le\nu^p
\qquad(t\ge0).
\]
Thus the dependence on the resolved frequency radius $R_*$, and hence on any prescribed $\nu$-dependent frequency scale $R_*\asymp\nu^{-p}$, cannot be improved beyond universal constants. In particular, under the natural heat scaling $p=1$, the example has $M=0$, $K_{0,\nu}=0$, $J_{*,\nu}\lesssim\nu^{-1}$, $c_{2,\nu}\lesssim\nu^{-1}$, and the actual ratio is at most $\nu$ for all $t\ge0$.

\textbf{Dependencies:} STEP8.
\end{proof}

\begin{remark}
This proof was produced by QED, supported by a Codex coding agent integrated with ChatGPT 5.5.
\end{remark}

\section{Fast time-periodic oscillations}\label{sec:fast}
\noindent
In the previous sections we considered on the special structure of shear flows. We now 
return to general divergence-free velocity fields, but consider them to be time-periodic 
with a large frequency $A$ (that is, $u(At,x,y)$ with $u$ having period $L$). The fast 
oscillation averages the advection, restoring a spectral gap and blocking the 
mode-to-mode transfer that could lead to superexponential decay. We prove an exponential 
lower bound for all sufficiently large $A$, with constants that are explicit in the data. 
The method uses a finite-dimensional adjoint bundle constructed from the time-averaged 
operator and a fast-phase averaging estimate.
\begin{theorem}
Consider on $\mathbb{T}^2$,
\begin{equation}\label{eq:nonshearnu33}
    \partial_t \rho + u(At,x,y)\!\cdot\!\nabla \rho = \nu \Delta \rho, \qquad \rho(0)=\rho_0 .
\end{equation}
Assume $\rho_0\neq0$, $\rho_0\in C^\infty$, $\int\rho_0=0$, $0<\nu\ll1$.  The velocity $u$ is real-valued,
divergence-free, $u\in L^\infty_t W^{1,\infty}_{x,y}$, and $u(t+L,\cdot)=u(t,\cdot)$ (period $L$).

Then we have explicit admissible thresholds: an upper bound on a sufficient $A_0$ and an upper bound on the
corresponding exponent $c_A>0$ (i.e.\ an admissible value) so that for all $A>A_0$,
$\|\rho(t)\|_{L^2}\ge C e^{-c_A t}$ with $C>0$. The constants may depend on $\rho_0$, $u$, $L$ and $\nu$.
\end{theorem}
\begin{remark}
\noindent
\begin{enumerate}
\item The authors only edit the wording of this proof. To make it clear that the proof is done by AI, the authors keep all the format of AI's proof.
\item One may find that this proof is a bit tedious. The reason this proof is not as clear is the authors ask AI to optimize the estimate. To see a more reader-friendly proof with a relatively loose estimate, one can check \cite{QEDp5}.
\item From \cite{rowan2025superexponential}, one cannot expect for an arbitrary velocity field, the $L^2$ norm of the solution of advection-diffusion equation will only decay exponentially. However, for time independent flow, one could use the method of Keldysh theorem to prove it has such an exponental bound \cite{huang2025}. But if one move to the time period flow with period large, then it should be close to the time-depent flow. Hence, we can only expect the fast time-periodic flow to be close to the time-independent flow. 
\item As an easy corollary, one directly get that for time-independent flow, the $L^2$ norm of the solution has an exponential lower bound.
\item This proof is based on the application of Keldysh theorem and Floquet theory.
\item One may also check that instead of $L^{\infty}_tW_{x,y}^{1,\infty}$, if one only has $u\in L^{\infty}$, the argument still work. For the sake of complicated functional analysis argument, we only discuss the situation when $u$ is Lipschitz in space.
\end{enumerate}
\end{remark}
\begin{proof}
We work in the complexification
\[
H=L^2_0(\mathbb T^2;\mathbb C)=\left\{f\in L^2(\mathbb T^2;\mathbb C):\int_{\mathbb T^2}f=0\right\}.
\]
The real solution is included by complexification. The norm $\|\cdot\|_2$ is the usual Hermitian
$L^2$ norm. Pairings with adjoint test functions use the bilinear pairing
\[
\langle f,g\rangle=\int_{\mathbb T^2}f(x)g(x)\,dx,
\]
which is the pairing naturally preserved by the forward equation and its adjoint. For a vector
$V=(v_1,\ldots,v_d)$ with entries in a Banach space $X$, write
\[
\|V\|_{X^d}=\left(\sum_{j=1}^d\|v_j\|_X^2\right)^{1/2}.
\]
Thus $H^d$ below means the $d$-fold product of $H=L^2_0(\mathbb T^2;\mathbb C)$; Sobolev
regularity is always written explicitly as $H^s(\mathbb T^2)^d$. All matrix norms are operator
norms on the relevant Euclidean space. Let $\lambda_1>0$ be the first eigenvalue of $-\Delta$ on
mean-zero functions on the chosen torus normalization.

\subsubsection*{STEP1: Energy, periodicity, and adjoint observables}
\noindent\\
\textbf{Claim:} Let $H=L^2_0(\mathbb T^2;\mathbb C)$, $b_A(t,x)=u(At,x)$, and $T_A=L/A$. For every $A>0$, the solution operator $U_A(t,s):H\to H$ for
\[
\partial_t\rho=\nu\Delta\rho-b_A(t)\cdot\nabla\rho
\]
satisfies, for all $0\le s\le t$,
\[
\int_{\mathbb T^2}U_A(t,s)f\,dx=0,
\qquad
\|U_A(t,s)f\|_2\le \|f\|_2,
\]
and, for smooth data,
\[
\frac12\frac{d}{dt}\|U_A(t,s)f\|_2^2+\nu\|\nabla U_A(t,s)f\|_2^2=0.
\]
The propagator is $T_A$-periodic in the sense
\[
U_A(t+T_A,s+T_A)=U_A(t,s).
\]
Moreover, if a vector-valued adjoint observable $\Phi_A(t)=(\phi_{A,1}(t),\ldots,\phi_{A,d}(t))$ satisfies
\[
\partial_t\Phi_A=-(\nu\Delta+b_A(t)\cdot\nabla)\Phi_A+\Phi_A G_A
\]
distributionally, then
\[
q_A(t):=\bigl(\langle U_A(t,0)\rho_0,\phi_{A,1}(t)\rangle,\ldots,\langle U_A(t,0)\rho_0,\phi_{A,d}(t)\rangle\bigr)
\]
satisfies
\[
q_A'(t)=G_A^{\mathsf T}q_A(t)
\]
in distributions.
\noindent\\
\textbf{Proof:}
For smooth data, integrate the equation over the torus:
\[
\frac{d}{dt}\int_{\mathbb T^2}\rho
=\nu\int_{\mathbb T^2}\Delta\rho-\int_{\mathbb T^2}b_A\cdot\nabla\rho.
\]
The Laplacian integral is zero by periodicity. Since $\nabla\cdot b_A=0$,
\[
\int_{\mathbb T^2}b_A\cdot\nabla\rho
=\int_{\mathbb T^2}\nabla\cdot(b_A\rho)-\int_{\mathbb T^2}(\nabla\cdot b_A)\rho=0.
\]
Thus the mean is preserved.

Multiplying by $\overline{\rho}$ and taking the real part gives
\[
\frac12\frac{d}{dt}\|\rho(t)\|_2^2
=\nu\operatorname{Re}\int_{\mathbb T^2}\overline{\rho}\Delta\rho
-\operatorname{Re}\int_{\mathbb T^2}\overline{\rho}\,b_A\cdot\nabla\rho.
\]
The diffusion term is $-\nu\|\nabla\rho\|_2^2$. The transport term vanishes because
\[
2\operatorname{Re}\int_{\mathbb T^2}\overline{\rho}\,b_A\cdot\nabla\rho
=\int_{\mathbb T^2}b_A\cdot\nabla|\rho|^2
=\int_{\mathbb T^2}\nabla\cdot(b_A|\rho|^2)-\int_{\mathbb T^2}(\nabla\cdot b_A)|\rho|^2=0.
\]
Hence
\[
\frac12\frac{d}{dt}\|\rho(t)\|_2^2+\nu\|\nabla\rho(t)\|_2^2=0.
\]
This is the standard incompressible advection-diffusion energy identity.

For general $L^2$ mean-zero data, take Fourier-Galerkin approximations. The displayed identity is uniform on compact time intervals, gives boundedness in
\[
L^\infty_{\rm loc}([s,\infty);L^2)\cap L^2_{\rm loc}([s,\infty);H^1),
\]
and gives uniqueness by applying the same identity to the difference of two solutions. Passing to the weak limit gives a unique weak solution operator $U_A(t,s)$, and integrating the energy identity gives
\[
\|U_A(t,s)f\|_2^2+2\nu\int_s^t\|\nabla U_A(r,s)f\|_2^2\,dr=\|f\|_2^2.
\]
In particular $\|U_A(t,s)f\|_2\le\|f\|_2$.

Since
\[
b_A(t+T_A,x)=u(A t+L,x)=u(At,x)=b_A(t,x),
\]
the shifted function $r\mapsto U_A(r+T_A,s+T_A)f$ and the unshifted function
$r\mapsto U_A(r,s)f$ solve the same Cauchy problem. Uniqueness gives
\[
U_A(t+T_A,s+T_A)=U_A(t,s).
\]

It remains to prove the observable identity. Let $\rho(t)=U_A(t,0)\rho_0$. For smooth $\rho$ and smooth
$\Phi_A$, using the bilinear pairing,
\[
\frac{d}{dt}\langle\rho,\phi_{A,j}\rangle
=\langle\nu\Delta\rho-b_A\cdot\nabla\rho,\phi_{A,j}\rangle
+\langle\rho,\partial_t\phi_{A,j}\rangle.
\]
Periodic integration by parts and $\nabla\cdot b_A=0$ give
\[
\langle\nu\Delta\rho-b_A\cdot\nabla\rho,\phi_{A,j}\rangle
=\langle\rho,\nu\Delta\phi_{A,j}+b_A\cdot\nabla\phi_{A,j}\rangle.
\]
The adjoint-bundle equation says
\[
\partial_t\phi_{A,j}
=-(\nu\Delta+b_A\cdot\nabla)\phi_{A,j}+\sum_{k=1}^d\phi_{A,k}(G_A)_{kj}.
\]
The infinite-dimensional adjoint terms cancel, and therefore
\[
\frac{d}{dt}\langle\rho,\phi_{A,j}\rangle
=\sum_{k=1}^d(G_A)_{kj}\langle\rho,\phi_{A,k}\rangle.
\]
This is exactly $q_A'=G_A^{\mathsf T}q_A$. The distributional statement follows by applying the same computation to time-mollified Galerkin approximations and passing to the weak limit; all products are justified by
\[
\rho\in L^\infty_{\rm loc}L^2\cap L^2_{\rm loc}H^1,\qquad
\Phi_A\in L^\infty_{\rm loc}L^2\cap L^2_{\rm loc}H^2
\]
in the applications below.

\textbf{Dependencies:} None.

\subsubsection*{STEP2: Averaged adjoint root space detecting $\rho_0$ (KEY STEP)}
\noindent\\
\textbf{Claim:} Define
\[
\bar u(x)=\frac1L\int_0^L u(\theta,x)\,d\theta,\qquad
B_\nu^*=\nu\Delta+\bar u\cdot\nabla
\]
with domain $D(B_\nu^*)=H^2(\mathbb T^2)\cap H$. The generalized root spaces of $B_\nu^*$ are complete in $H$.

For the fixed $0\ne\rho_0\in C^\infty\cap H$, set
\[
\Sigma_\nu(\rho_0)=
\left\{\lambda\in\sigma(B_\nu^*):\exists m\ge1,\ \exists 0\ne\phi\in\ker(B_\nu^*-\lambda)^m,\ \langle\rho_0,\phi\rangle\ne0\right\}.
\]
Then $\Sigma_\nu(\rho_0)\ne\varnothing$, and since $B_\nu^*$ has compact resolvent one may choose
\[
\lambda_\nu\in\Sigma_\nu(\rho_0)
\quad\text{with}\quad
\gamma_\nu:=-\operatorname{Re}\lambda_\nu
=
\min_{\lambda\in\Sigma_\nu(\rho_0)}\bigl(-\operatorname{Re}\lambda\bigr)<\infty.
\]
Let $E_\nu=\bigcup_{m\ge1}\ker(B_\nu^*-\lambda_\nu)^m$, $d_\nu=\dim E_\nu$, and choose a basis
\[
\Phi_\nu=(\phi_1,\ldots,\phi_{d_\nu})\in(H^2(\mathbb T^2)\cap H)^{d_\nu}.
\]
There is a matrix $G_\nu\in\mathbb C^{d_\nu\times d_\nu}$ such that
\[
B_\nu^*\Phi_\nu=\Phi_\nu G_\nu,\qquad
\sigma(G_\nu)=\{\lambda_\nu\},
\]
and
\[
q_\nu:=
\left(\langle\rho_0,\phi_1\rangle,\ldots,\langle\rho_0,\phi_{d_\nu}\rangle\right)
\quad\text{satisfies}\quad
Q_\nu:=\|q_\nu\|_{\mathbb C^{d_\nu}}>0.
\]
Also define the finite constants
\[
K_{0,\nu}=\left(\sum_{j=1}^{d_\nu}\|\phi_j\|_2^2\right)^{1/2},\quad
K_{2,\nu}=\left(\sum_{j=1}^{d_\nu}\|\phi_j\|_{H^2}^2\right)^{1/2},\quad
g_\nu=\|G_\nu\|.
\]
\noindent
\textbf{Proof:}
\begin{key-original-step}
First note that $\bar u\in W^{1,\infty}(\mathbb T^2)$ and $\nabla\cdot\bar u=0$. Fix any $\alpha>0$ and set
\[
R=(\alpha-B_\nu^*)^{-1},\qquad R_0=(\alpha-\nu\Delta)^{-1}
\]
on $H$. We justify these objects and then apply Keldysh's theorem to $R$.

For $f\in H^2\cap H$,
\[
\operatorname{Re}( (\alpha-B_\nu^*)f,f)_{L^2}
=\alpha\|f\|_2^2+\nu\|\nabla f\|_2^2,
\]
because $\bar u\cdot\nabla$ is skew-symmetric in the Hermitian $L^2$ pairing. The associated coercive form on $H^1\cap H$ gives, by the Galerkin/Lax-Milgram argument, a unique $H^1$ solution of
$(\alpha-B_\nu^*)f=h$ for each $h\in H$. Since
\[
\nu\Delta f=-h+\alpha f-\bar u\cdot\nabla f\in L^2,
\]
elliptic regularity on the torus gives $f\in H^2\cap H$. Thus $R:H\to H^2\cap H$ is bounded.

The heat resolvent $R_0$ is compact, self-adjoint, injective, and of finite order. Indeed, if
$-\Delta e_k=\mu_k e_k$ is an orthonormal Fourier basis of mean-zero eigenfunctions, then
\[
R_0 e_k=(\alpha+\nu\mu_k)^{-1}e_k.
\]
In two dimensions $\sum_k(\alpha+\nu\mu_k)^{-p}<\infty$ for every $p>1$, so $R_0$ belongs to a Schatten class and is of finite order. This is the concrete Fourier instance of the compact self-adjoint spectral theorem, see \cite[Theorem 1.1]{shlapunov2013completeness}.

The resolvent identity is
\[
(\alpha-\nu\Delta)Rh=h+\bar u\cdot\nabla Rh,
\]
so
\[
Rh=R_0(I+\delta A)h,\qquad \delta A h=\bar u\cdot\nabla Rh.
\]
Because $R:H\to H^2$ is bounded, $\nabla R:H\to H^1$ is bounded. Multiplication by
$\bar u\in W^{1,\infty}$ maps $H^1$ boundedly into $H^1$, and the embedding
$H^1(\mathbb T^2)\hookrightarrow L^2(\mathbb T^2)$ is compact. Hence $\delta A:H\to H$ is compact. Also $R=R_0(I+\delta A)$ is injective because it is a resolvent.

Keldysh's weak compact perturbation theorem applies (see \cite[Theorem 1.4]{shlapunov2013completeness}).

The root spaces of $R$ and $B_\nu^*$ coincide after the spectral change
\[
\mu=(\alpha-\lambda)^{-1}.
\]
Indeed,
\[
R-\mu I=\mu R\bigl(B_\nu^*-\lambda I\bigr),
\]
and $R$ commutes with $B_\nu^*-\lambda I$ on the relevant domains. Iterating this identity gives
\[
\ker(R-\mu I)^m=\ker(B_\nu^*-\lambda I)^m.
\]
Therefore the generalized root spaces of $B_\nu^*$ are complete in $H$.

Now suppose, for contradiction, that $\rho_0$ annihilates every generalized root vector of $B_\nu^*$ under the bilinear pairing. The root vectors have dense linear span in $H$, so the continuous linear functional
\[
\phi\mapsto \langle\rho_0,\phi\rangle
\]
vanishes on all of $H$. Taking $\phi=\overline{\rho_0}$ gives $\|\rho_0\|_2^2=0$, contradicting $\rho_0\ne0$. Hence $\Sigma_\nu(\rho_0)\ne\varnothing$.

The operator $B_\nu^*$ has compact resolvent because $R$ is compact. Its eigenvalues have finite algebraic multiplicity and no finite accumulation point, since the nonzero spectral values of the compact operator $R$ have that property and $\lambda=\alpha-\mu^{-1}$. We also need the minimum defining $\gamma_\nu$ to be attained. Let $\lambda$ be an eigenvalue of $B_\nu^*$ with eigenvector $f\ne0$. Then
\[
\operatorname{Re}\lambda
=-\nu\frac{\|\nabla f\|_2^2}{\|f\|_2^2}
\le -\nu\lambda_1.
\]
If $-\operatorname{Re}\lambda\le R_*$, then $\|\nabla f\|_2/\|f\|_2\le (R_*/\nu)^{1/2}$, and
\[
|\operatorname{Im}\lambda|
\le \|\bar u\|_\infty\frac{\|\nabla f\|_2}{\|f\|_2}
\le \|\bar u\|_\infty (R_*/\nu)^{1/2}.
\]
Thus all eigenvalues with $-\operatorname{Re}\lambda\le R_*$ lie in a bounded subset of $\mathbb C$. Since the spectrum has no finite accumulation and each eigenvalue has finite algebraic multiplicity, there are only finitely many such eigenvalues. Choosing any detecting eigenvalue first and taking $R_*$ to be its $-\operatorname{Re}$ value shows that the detecting set has an attained minimum. This gives $\lambda_\nu$ and $\gamma_\nu<\infty$.

Let $E_\nu$ be the full generalized root space for $\lambda_\nu$. It is finite-dimensional and invariant under $B_\nu^*$. Therefore, in any basis
$\Phi_\nu=(\phi_1,\ldots,\phi_{d_\nu})$ of $E_\nu$, there is a matrix $G_\nu$ satisfying
\[
B_\nu^*\Phi_\nu=\Phi_\nu G_\nu.
\]
Because $E_\nu$ is the generalized root space for the single spectral value $\lambda_\nu$, the finite-dimensional matrix representing $B_\nu^*|_{E_\nu}$ has spectrum $\{\lambda_\nu\}$. Root vectors belong to the domain of $B_\nu^*$, hence to $H^2\cap H$. Finally, since $\lambda_\nu\in\Sigma_\nu(\rho_0)$, the functional $\phi\mapsto\langle\rho_0,\phi\rangle$ is not identically zero on $E_\nu$. Therefore its coordinate vector in the chosen basis satisfies $Q_\nu>0$. The constants $K_{0,\nu},K_{2,\nu},g_\nu$ are finite because $d_\nu<\infty$ and $\phi_j\in H^2$.
\end{key-original-step}
\begin{heuristics}
The averaged adjoint is an elliptic dissipative operator plus a skew first-order perturbation. Its resolvent is a compact perturbation of the self-adjoint heat resolvent in exactly the weak Keldysh sense, so the root vectors cannot miss a nonzero datum. The choice of the detecting root space with minimal $-\operatorname{Re}\lambda$ is legitimate because only finitely many eigenvalues can live above any fixed dissipative level.
\end{heuristics}

\textbf{Dependencies:} STEP1.

\subsubsection*{STEP3: Fast-periodic persistence of the detecting adjoint bundle (KEY STEP)}
\noindent\\
\textbf{Claim:} Let $M=1+\|u\|_{L^\infty([0,L];W^{1,\infty})}$. Let $C_R(L)\ge1$ be a constant such that for every $0<\delta\le1$, every zero-phase $F\in L^2(\mathbb R/L\mathbb Z;H)$, and
\[
\widehat{R_\delta F}_{\ell m}
=
\frac{\widehat F_{\ell m}}{i(2\pi\ell/L)-\delta\mu_m}
\qquad(\ell\ne0,\ -\Delta e_m=\mu_m e_m),
\]
one has
\[
\|R_\delta F\|_{C^0_\theta H}
+\sqrt{\delta}\|\nabla R_\delta F\|_{L^2_\theta L^2_x}
+\delta\|\Delta R_\delta F\|_{L^2_\theta L^2_x}
+\|\partial_\theta R_\delta F\|_{L^2_\theta H}
\le C_R(L)\|F\|_{L^2_\theta H}.
\]
Let $P_\nu$ be the Riesz projection onto $E_\nu$, and write $\Pi_\nu^\perp=I-P_\nu$. Let $C_{S,\nu}\ge1$ be a finite Sylvester-resolvent constant for
\[
\mathcal L_\nu(Y,\mathcal H)=-B_\nu^*Y+YG_\nu+\Phi_\nu \mathcal H,\qquad P_\nu Y=0,
\]
namely
\[
\|Y\|_{H^1}+\|\mathcal H\|\le C_{S,\nu}\|\mathcal R\|_{H^{-1}}
\]
and
\[
\|Y\|_{H^2}+\|\mathcal H\|\le C_{S,\nu}\|\mathcal R\|_{L^2}
\]
whenever $\mathcal L_\nu(Y,\mathcal H)=\mathcal R$ and $\mathcal R$ belongs to the indicated space. Define
\[
S_\nu:=1+K_{2,\nu}+g_\nu,
\qquad
\mathcal K_\nu
=
10^6\,C_R(L)^2\,C_{S,\nu}^2\,M^2S_\nu^2,
\]
and the sufficient bundle threshold
\[
A_{\mathrm{bun},\nu}
=
\max\left\{4\mathcal K_\nu,\ \nu,\ \frac{64\mathcal K_\nu^2}{\nu},\ 1000\,C_{S,\nu}\mathcal K_\nu\right\}.
\]
For every $A>A_{\mathrm{bun},\nu}$, there exist a $T_A=L/A$-periodic family
\[
\Phi_A(t)=(\phi_{A,1}(t),\ldots,\phi_{A,d_\nu}(t))
\]
and a matrix $G_A\in\mathbb C^{d_\nu\times d_\nu}$ satisfying
\[
\partial_t\Phi_A
=
-(\nu\Delta+b_A(t)\cdot\nabla)\Phi_A+\Phi_A G_A
\]
in $L^2_{\rm loc}(\mathbb R;H^{-2})$, and the quantitative estimates
\[
\sup_{t\in\mathbb R}\|\Phi_A(t)-\Phi_\nu\|_{H^{d_\nu}}
+\|G_A-G_\nu\|
\le \frac{\mathcal K_\nu}{A},
\]
\[
\sup_{t\in\mathbb R}\|\Phi_A(t)\|_{H^{d_\nu}}
\le K_{0,\nu}+1,
\]
and
\[
\left\|
\bigl(\langle\rho_0,\phi_{A,1}(0)\rangle,\ldots,\langle\rho_0,\phi_{A,d_\nu}(0)\rangle\bigr)
-q_\nu
\right\|
\le \frac{\mathcal K_\nu\|\rho_0\|_2}{A}.
\]
Consequently, if also
\[
A>A_{\mathrm{pair},\nu}:=\frac{2\mathcal K_\nu\|\rho_0\|_2}{Q_\nu},
\]
then the initial nonautonomous observable satisfies
\[
\|q_A(0)\|\ge \frac12Q_\nu.
\]
\noindent
\textbf{Proof:}
\begin{key-original-step}
We first verify the two auxiliary constants used in the statement.

For $R_\delta$, put $\omega_\ell=2\pi\ell/L$. The Fourier multiplier denominator is
\[
i\omega_\ell-\delta\mu_m,\qquad \ell\ne0.
\]
The inequalities
\[
\left|\frac1{i\omega_\ell-\delta\mu_m}\right|\le\frac1{|\omega_\ell|},\qquad
\left|\frac{\omega_\ell}{i\omega_\ell-\delta\mu_m}\right|\le1,\qquad
\left|\frac{\delta\mu_m}{i\omega_\ell-\delta\mu_m}\right|\le1
\]
give the $L^2_\theta H$, $\partial_\theta$, and $\delta\Delta$ bounds. Also
\[
\frac{\sqrt{\delta}\mu_m^{1/2}}{(\omega_\ell^2+\delta^2\mu_m^2)^{1/2}}
\le \frac{1}{\sqrt{2|\omega_\ell|}}
\le \sqrt{\frac{L}{4\pi}},
\]
where the first inequality follows by maximizing $x^{1/2}/(a^2+x^2)^{1/2}$ with
$x=\delta\mu_m$ and $a=|\omega_\ell|$. The $C^0_\theta H$ bound follows from the one-dimensional embedding
\[
H^1(\mathbb R/L\mathbb Z;H)\hookrightarrow C^0(\mathbb R/L\mathbb Z;H).
\]
Thus one may take any $C_R(L)\ge1$ large enough to dominate these elementary constants.

We next justify that $C_{S,\nu}$ is finite. For $s=-1,0$, regard $B_\nu^*$ as the closed elliptic operator
\[
B_\nu^*:H^{s+2}_0(\mathbb T^2)\to H^s_0(\mathbb T^2),
\]
where $H^s_0$ denotes the mean-zero Sobolev space, and $H^{-1}_0=(H^1_0)^*$. Let $\Gamma$ be a positively oriented contour enclosing only the spectral point $\lambda_\nu$. The Riesz projection
\[
P_\nu=\frac{1}{2\pi i}\int_\Gamma (z-B_\nu^*)^{-1}\,dz
\]
is bounded on both $H^{-1}_0$ and $L^2_0$. This follows from the elliptic estimates
\[
\|(z-B_\nu^*)^{-1}f\|_{H^{s+2}}\le C_{\Gamma,s,\nu}\|f\|_{H^s},\qquad z\in\Gamma,\ s=-1,0.
\]
Indeed, if such an estimate failed for a fixed compact contour, there would be $z_n\in\Gamma$ and $w_n$ with $\|w_n\|_{H^{s+2}}=1$ but $\|(z_n-B_\nu^*)w_n\|_{H^s}\to0$; compactness of the lower Sobolev embedding and the ellipticity of $\nu\Delta$ would give a nonzero $w$ solving $(z-B_\nu^*)w=0$ for some $z\in\Gamma$, contradicting that $\Gamma$ lies in the resolvent set.

Let $B_{\perp}=B_\nu^*|_{\Pi_\nu^\perp H}$. Its spectrum is disjoint from $\sigma(G_\nu)=\{\lambda_\nu\}$. Given a vector-valued right side $\mathcal R$, write $P_\nu\mathcal R=\Phi_\nu \mathbf{h}_{\mathcal R}$ in the chosen basis $\Phi_\nu$, column by column, and set
\[
\mathbf{h}=\mathbf{h}_{\mathcal R},\qquad
Y=\frac{1}{2\pi i}\int_\Gamma (z-B_{\perp})^{-1}\Pi_\nu^\perp\mathcal R\,(z-G_\nu)^{-1}\,dz .
\]
The coordinate map $\mathcal R\mapsto \mathbf{h}_{\mathcal R}$ is bounded from $H^{-1}$ and from $L^2$ because $P_\nu$ has finite-dimensional range. The contour formula gives
\[
B_{\perp}Y-YG_\nu=-\Pi_\nu^\perp\mathcal R.
\]
The identity is obtained by inserting
\[
B_{\perp}(z-B_{\perp})^{-1}=z(z-B_{\perp})^{-1}-I,\qquad
(z-G_\nu)^{-1}G_\nu=z(z-G_\nu)^{-1}-I
\]
under the integral. The integral of $(z-B_{\perp})^{-1}$ over $\Gamma$ is zero, while the integral of $(z-G_\nu)^{-1}$ is $2\pi i I$. Hence
\[
-B_\nu^*Y+YG_\nu+\Phi_\nu \mathbf{h}=\Pi_\nu^\perp\mathcal R+P_\nu\mathcal R=\mathcal R,\qquad P_\nu Y=0.
\]
The same contour estimates give
\[
\|Y\|_{H^1}+\|\mathcal H\|\le C_{S,\nu}\|\mathcal R\|_{H^{-1}},
\qquad
\|Y\|_{H^2}+\|\mathcal H\|\le C_{S,\nu}\|\mathcal R\|_{L^2}
\]
after increasing one finite constant $C_{S,\nu}\ge1$.

Now set
\[
\varepsilon=A^{-1},\qquad \delta=\varepsilon\nu,\qquad \Theta=\mathbb R/L\mathbb Z,
\]
and seek
\[
\Phi_A(t)=\Psi_\varepsilon(At),\qquad G_A=G_\varepsilon.
\]
The desired equation is equivalent to the $L$-periodic phase equation
\begin{equation}\label{eq:phase}
\partial_\theta\Psi_\varepsilon
=\varepsilon\left[-(\nu\Delta+u(\theta)\cdot\nabla)\Psi_\varepsilon
+\Psi_\varepsilon G_\varepsilon\right].
\end{equation}
Let $\mathcal M F=L^{-1}\int_0^L F(\theta)\,d\theta$. Write
\begin{equation}\label{eq:decomp}
\Psi_\varepsilon=\Phi_\nu+Y+Z,\qquad \mathcal MZ=0,\qquad P_\nu Y=0,\qquad
G_\varepsilon=G_\nu+\mathcal H.
\end{equation}
Putting $\partial_\theta+\varepsilon\nu\Delta$ on the left of the zero-phase equation gives
\begin{equation}\label{eq:Zdef}
Z=\varepsilon R_\delta F_0(Z,Y,\mathcal H),
\end{equation}
where
\begin{equation}	\label{eq:F0}
F_0(Z,Y,\mathcal H)
=-(u-\bar u)\cdot\nabla(\Phi_\nu+Y)
+(I-\mathcal M)\{-u\cdot\nabla Z+Z(G_\nu+\mathcal H)\}.
\end{equation}
The spatial mean of $F_0$ is zero because $u(\theta)$, $\bar u$, and $u(\theta)-\bar u$ are divergence-free. Its phase mean is zero by construction. Taking the phase mean of \eqref{eq:phase} gives
\begin{equation}\label{eq:mean}
\mathcal L_\nu(Y,\mathcal H)
=\mathcal M\{u\cdot\nabla Z-ZG_\nu-Z\mathcal H\}-Y\mathcal H.
\end{equation}
Equations \eqref{eq:Zdef} and \eqref{eq:mean} are exactly equivalent to \eqref{eq:phase}, since their sum reconstructs the zero-phase and phase-mean parts of \eqref{eq:phase}.

Let
\[
r:=\mathcal K_\nu\varepsilon.
\]
Since $A>A_{\rm bun,\nu}$,
\begin{equation}\label{eq:small}
\delta\le1,\qquad
r<\frac14,\qquad
\sqrt{\frac{\varepsilon}{\nu}}\le \frac{1}{8\mathcal K_\nu},\qquad
C_{S,\nu}r\le\frac1{1000}.
\end{equation}
We use normalized $L^2_\theta$ norms on $\Theta$. Consider triples $(Z,Y,\mathcal H)$ with
$\mathcal MZ=0$, $P_\nu Y=0$, and finite norm
\begin{equation}\label{eq:norm}
\|(Z,Y,\mathcal H)\|_\varepsilon
=\|Z\|_{C^0_\theta H^{d_\nu}}
+\sqrt\delta\|\nabla Z\|_{L^2_\theta L^2_x}
+\delta\|\Delta Z\|_{L^2_\theta L^2_x}
+\|Y\|_{(L^2)^{d_\nu}}
+\sqrt\delta\|Y\|_{(H^2)^{d_\nu}}
+\|\mathcal H\|.
\end{equation}
This is a complete metric space after identifying functions equal almost everywhere.

Define a map $\mathfrak T$ on the closed ball $\|(Z,Y,\mathcal H)\|_\varepsilon\le r$ by
\begin{equation}\label{eq:Zplus}
Z^+=\varepsilon R_\delta F_0(Z,Y,\mathcal H),
\end{equation}
and by letting $(Y^+,\mathcal H^+)$ be the Sylvester solution of
\begin{equation}\label{eq:Sylv}
\mathcal L_\nu(Y^+,\mathcal H^+)
=\mathcal M\{u\cdot\nabla Z^+-Z^+G_\nu-Z^+\mathcal H\}-Y\mathcal H.
\end{equation}
We prove that $\mathfrak T$ maps the ball into itself and is a contraction.

The elementary product estimates used repeatedly are, for every divergence-free $v\in W^{1,\infty}$,
\begin{equation}	\label{eq:prodest}
\|v\cdot\nabla W\|_{H^{-1}}\le \|v\|_\infty\|W\|_2,
\qquad
\|v\cdot\nabla W\|_2\le \|v\|_\infty\|\nabla W\|_2.
\end{equation}
The first inequality follows by testing against $a\in H^1$ with $\|a\|_{H^1}\le1$:
\[
|\langle v\cdot\nabla W,a\rangle|
=|\langle W,v\cdot\nabla a\rangle|
\le \|v\|_\infty\|W\|_2.
\]
On the ball, \eqref{eq:small} gives
\begin{equation}\label{eq:Yball}
\|Y\|_{H^2}\le\frac{r}{\sqrt\delta}
=\mathcal K_\nu\sqrt{\frac{\varepsilon}{\nu}}\le\frac18,
\qquad
\|\nabla Z\|_{L^2_\theta L^2_x}\le\frac18.
\end{equation}
Using $\|u-\bar u\|_{L^\infty_\theta W^{1,\infty}}\le2M$, \eqref{eq:F0}, and $r<1$,
\begin{equation}\label{eq:F0est}
\begin{aligned}
\|F_0(Z,Y,\mathcal H)\|_{L^2_\theta H^{d_\nu}}
&\le 2M(K_{2,\nu}+\|Y\|_{H^2})
+M\|\nabla Z\|_{L^2_\theta L^2_x}
+(g_\nu+\|\mathcal H\|)\|Z\|_{L^2_\theta H}  \\
&\le 4MS_\nu .
\end{aligned}
\end{equation}
Therefore
\begin{equation}\label{eq:Zplusest}
\|Z^+\|_{C^0_\theta H}
+\sqrt\delta\|\nabla Z^+\|_{L^2_\theta L^2_x}
+\delta\|\Delta Z^+\|_{L^2_\theta L^2_x}
\le 4C_R(L)MS_\nu\varepsilon
\le \frac{r}{100}.
\end{equation}
The last inequality follows from $\mathcal K_\nu=10^6C_R(L)^2C_{S,\nu}^2M^2S_\nu^2$ and all factors being at least one.

Let
\[
\mathcal R_{\rm mean}
=\mathcal M\{u\cdot\nabla Z^+-Z^+G_\nu-Z^+\mathcal H\}-Y\mathcal H.
\]
Using \eqref{eq:prodest}, \eqref{eq:Zplusest}, $M+g_\nu+1\le2MS_\nu$, and $C_{S,\nu}r\le10^{-3}$, we get
\begin{equation}\label{eq:RmHm1}
\begin{aligned}
\|\mathcal R_{\rm mean}\|_{H^{-1}}
&\le (M+g_\nu+\|\mathcal H\|)\|Z^+\|_{L^2_\theta H}+\|Y\|_2\|\mathcal H\|  \\
&\le 8C_R(L)M^2S_\nu^2\varepsilon+r^2.
\end{aligned}
\end{equation}
The Sylvester $H^{-1}$ bound gives
\begin{equation}\label{eq:YmH1}
\|Y^+\|_{H^1}+\|\mathcal H^+\|
\le C_{S,\nu}\left(8C_R(L)M^2S_\nu^2\varepsilon+r^2\right)
\le \frac{r}{50}.
\end{equation}
Indeed, the first term is at most $8\cdot10^{-6}r$, and the second is at most $10^{-3}r$.

For the $H^2$ part, use the $L^2$ estimate in \eqref{eq:prodest}:
\[
\|\mathcal R_{\rm mean}\|_{L^2}
\le M\|\nabla Z^+\|_{L^2_\theta L^2_x}
+(g_\nu+\|\mathcal H\|)\|Z^+\|_{L^2_\theta H}
+\|Y\|_2\|\mathcal H\|.
\]
Multiplying by $\sqrt\delta$, using $\sqrt\delta\le1$, and using \eqref{eq:Zplusest} gives
\begin{equation}\label{eq:RmL2}
\sqrt\delta\,\|\mathcal R_{\rm mean}\|_{L^2}
\le 8C_R(L)M^2S_\nu^2\varepsilon+r^2.
\end{equation}
Hence
\begin{equation}\label{eq:YmH2}
\sqrt\delta\|Y^+\|_{H^2}+\sqrt\delta\|\mathcal H^+\|
\le \frac{r}{50}.
\end{equation}
The norm contains $\|\mathcal H^+\|$ without $\sqrt\delta$, and this is already controlled by \eqref{eq:YmH1}. Equations \eqref{eq:Zplusest}, \eqref{eq:YmH1}, and \eqref{eq:YmH2} imply $\|\mathfrak T(Z,Y,\mathcal H)\|_\varepsilon\le r$.

Now compare two triples in the ball and write differences with a prefix $\mathrm d$. From \eqref{eq:F0},
\begin{equation}\label{eq:dF0}
\begin{aligned}
\|\mathrm d F_0\|_{L^2_\theta H}
&\le 2M\|\mathrm dY\|_{H^2}
+M\|\nabla\mathrm dZ\|_{L^2_\theta L^2_x}
+(g_\nu+1)\|\mathrm dZ\|_{L^2_\theta H}
+\|\mathrm d\mathcal H\|\,\|Z_2\|_{L^2_\theta H}  \\
&\le 6MS_\nu\,\delta^{-1/2}\|(\mathrm dZ,\mathrm dY,\mathrm d\mathcal H)\|_\varepsilon.
\end{aligned}
\end{equation}
Thus
\begin{equation}\label{eq:dZplus}
\begin{aligned}
&\|\mathrm dZ^+\|_{C^0_\theta H}
+\sqrt\delta\|\nabla\mathrm dZ^+\|_{L^2_\theta L^2_x}
+\delta\|\Delta(\mathrm dZ^+)\|_{L^2_\theta L^2_x}  \\
&\qquad\le
6C_R(L)MS_\nu\sqrt{\frac{\varepsilon}{\nu}}\,
\|(\mathrm dZ,\mathrm dY,\mathrm d\mathcal H)\|_\varepsilon
\le \frac1{100}\|(\mathrm dZ,\mathrm dY,\mathrm d\mathcal H)\|_\varepsilon .
\end{aligned}
\end{equation}
The last inequality follows from \eqref{eq:small} and the definition of $\mathcal K_\nu$.

For the mean equation, the difference of the right side of \eqref{eq:Sylv} is a sum of
\[
\mathcal M(u\cdot\nabla\mathrm dZ^+),\quad
-\mathcal M(\mathrm dZ^+G_\nu),\quad
-\mathcal M(\mathrm dZ^+\mathcal H_1),\quad
-\mathcal M(Z_2\,\mathrm d\mathcal H),\quad
-\mathrm dY\,\mathcal H_1,\quad
-Y_2\,\mathrm d\mathcal H.
\]
Its $H^{-1}$ norm is bounded by
\[
(M+g_\nu+1)\|\mathrm dZ^+\|_{L^2_\theta H}
+r\|\mathrm d\mathcal H\|+r\|\mathrm dY\|_2+r\|\mathrm d\mathcal H\|.
\]
Its $L^2$ norm, after multiplication by $\sqrt\delta$, is bounded by the same expression with
$\|\mathrm dZ^+\|_{L^2_\theta H}$ replaced by the left side of \eqref{eq:dZplus}. Set
\[
\alpha=6C_R(L)MS_\nu\sqrt{\frac{\varepsilon}{\nu}}.
\]
Then the pre-final inequality in \eqref{eq:dZplus} gives the $Z^+$-difference norm bounded by
$\alpha\|(\mathrm dZ,\mathrm dY,\mathrm d\mathcal H)\|_\varepsilon$. Since
\[
C_{S,\nu}(M+g_\nu+1)\alpha
\le
12C_R(L)C_{S,\nu}M^2S_\nu^2\sqrt{\frac{\varepsilon}{\nu}}
\le
\frac{12C_R(L)C_{S,\nu}M^2S_\nu^2}{8\mathcal K_\nu}
\le 2\cdot10^{-6},
\]
and $3C_{S,\nu}r\le3\cdot10^{-3}$, applying the two Sylvester estimates gives
\begin{equation}\label{eq:dYm}
\|\mathrm dY^+\|_{H^1}+\|\mathrm d\mathcal H^+\|+\sqrt\delta\|\mathrm dY^+\|_{H^2}
\le \frac1{50}\|(\mathrm dZ,\mathrm dY,\mathrm d\mathcal H)\|_\varepsilon.
\end{equation}
Combining \eqref{eq:dZplus} and \eqref{eq:dYm}, $\mathfrak T$ is a strict contraction, with Lipschitz constant less than $1/10$. The fixed-point theorem gives a unique fixed point in the ball.

At the fixed point, \eqref{eq:Zdef} and \eqref{eq:mean} imply \eqref{eq:phase} in $L^2(\Theta;H^{-2})$. Define
\[
\Phi_A(t)=\Psi_{1/A}(At),\qquad G_A=G_{1/A}.
\]
The $L$-periodicity of $\Psi_{1/A}$ gives $T_A=L/A$ periodicity of $\Phi_A$, and multiplying \eqref{eq:phase} by $A$ gives
\[
\partial_t\Phi_A
=-(\nu\Delta+b_A(t)\cdot\nabla)\Phi_A+\Phi_A G_A
\]
in $L^2_{\rm loc}(\mathbb R;H^{-2})$.

Because the fixed point equals its image under $\mathfrak T$, the sharper output estimates \eqref{eq:Zplusest}, \eqref{eq:YmH1}, and \eqref{eq:YmH2} apply to it. Thus
\[
\sup_t\|\Phi_A(t)-\Phi_\nu\|_{H^{d_\nu}}+\|G_A-G_\nu\|
\le \frac{\mathcal K_\nu}{A}.
\]
Also $A>4\mathcal K_\nu$ gives $\mathcal K_\nu/A<1$, so
\[
\sup_t\|\Phi_A(t)\|_{H^{d_\nu}}\le K_{0,\nu}+1.
\]
Finally, by Cauchy's inequality in each coordinate,
\[
\left\|
\bigl(\langle\rho_0,\phi_{A,1}(0)\rangle,\ldots,\langle\rho_0,\phi_{A,d_\nu}(0)\rangle\bigr)
-q_\nu
\right\|
\le
\|\rho_0\|_2\,\|\Phi_A(0)-\Phi_\nu\|_{H^{d_\nu}}
\le \frac{\mathcal K_\nu\|\rho_0\|_2}{A}.
\]
If $A>A_{\rm pair,\nu}=2\mathcal K_\nu\|\rho_0\|_2/Q_\nu$, this last inequality gives
\[
\|q_A(0)\|\ge Q_\nu-\frac{\mathcal K_\nu\|\rho_0\|_2}{A}\ge \frac12Q_\nu.
\]

The construction above is a direct finite-cluster averaging argument. It is consistent with the general fast periodic parabolic averaging principle of Matthies \cite{matthies2001time}.
\end{key-original-step}
\begin{heuristics}
The nonzero phase modes are controlled by $(\partial_\theta+\varepsilon\nu\Delta)^{-1}$. The denominator $i\omega_\ell-\varepsilon\nu\mu_m$ prevents high spatial frequencies from destroying the averaging estimate, but derivative terms still lose $\sqrt{\varepsilon/\nu}$, explaining the $64\mathcal K_\nu^2/\nu$ threshold. The mean equation is finite-codimensional: the chosen root space supplies the missing matrix parameter $\mathcal H=G_\varepsilon-G_\nu$, and the complement is solved by a Sylvester resolvent. The added threshold $1000C_{S,\nu}\mathcal K_\nu$ makes the quadratic mean term $Y\mathcal H$ genuinely perturbative.
\end{heuristics}

\textbf{Dependencies:} STEP2.

\subsubsection*{STEP4: Finite-dimensional lower bound from the adjoint bundle}
\noindent\\
\textbf{Claim:} Fix any tuning parameter $0<\eta\le1$. Define
\[
D_{\nu,\eta}:=
\sup_{t\ge0}
e^{-(\gamma_\nu+\eta)t}
\|e^{-G_\nu^{\mathsf T}t}\|_{\mathbb C^{d_\nu}\to\mathbb C^{d_\nu}}<\infty.
\]
Equivalently, if $G_\nu^{\mathsf T}=S_\nu^{J}(\lambda_\nu I+N_\nu)(S_\nu^{J})^{-1}$, $N_\nu^{r_\nu}=0$, then one may use the explicit upper bound
\[
D_{\nu,\eta}
\le
\|S_\nu^{J}\|\,\|(S_\nu^{J})^{-1}\|
\sum_{k=0}^{r_\nu-1}\eta^{-k}.
\]
For every
\[
A>A_{\nu}^{(0)}:=
\max\{A_{\mathrm{bun},\nu},A_{\mathrm{pair},\nu}\},
\]
the solution $\rho_A(t)=U_A(t,0)\rho_0$ satisfies
\[
\|\rho_A(t)\|_2
\ge
C_{\nu,\eta}\exp[-c_{A,\nu,\eta}t]
\qquad(t\ge0),
\]
with
\[
C_{\nu,\eta}
=
\frac{Q_\nu}{2D_{\nu,\eta}(K_{0,\nu}+1)}
\]
and
\[
c_{A,\nu,\eta}
=
\gamma_\nu+\eta
+
\frac{D_{\nu,\eta}\mathcal K_\nu}{A}.
\]
\noindent
\textbf{Proof:}
Since $\sigma(G_\nu)=\{\lambda_\nu\}$ and $G_\nu$ is finite-dimensional, $G_\nu^{\mathsf T}$ is similar to
$\lambda_\nu I+N_\nu$, where $N_\nu$ is nilpotent. Hence
\[
e^{-G_\nu^{\mathsf T}t}
=S_\nu^{J} e^{-\lambda_\nu t}e^{-N_\nu t}(S_\nu^{J})^{-1}.
\]
Choosing Jordan form for $N_\nu$, $\|N_\nu^k\|\le1$ for $0\le k<r_\nu$, and
\[
\begin{aligned}
e^{-(\gamma_\nu+\eta)t}\|e^{-G_\nu^{\mathsf T}t}\|
&\le \|S_\nu^{J}\|\|(S_\nu^{J})^{-1}\|e^{-\eta t}
\sum_{k=0}^{r_\nu-1}\frac{t^k}{k!}\|N_\nu^k\|  \\
&\le \|S_\nu^{J}\|\|(S_\nu^{J})^{-1}\|\sum_{k=0}^{r_\nu-1}\eta^{-k}.
\end{aligned}
\]
Thus $D_{\nu,\eta}<\infty$, and the displayed explicit bound is valid.

Let $E_A=G_A^{\mathsf T}-G_\nu^{\mathsf T}$. STEP3 gives
\[
\|E_A\|\le \frac{\mathcal K_\nu}{A}.
\]
Duhamel's formula for $e^{-(G_\nu^{\mathsf T}+E_A)t}$ gives
\[
e^{-G_A^{\mathsf T}t}
=e^{-G_\nu^{\mathsf T}t}
-\int_0^t e^{-G_\nu^{\mathsf T}(t-s)}E_Ae^{-G_A^{\mathsf T}s}\,ds.
\]
Using the definition of $D_{\nu,\eta}$ and Gronwall's inequality,
\begin{equation}\label{eq:41}
\|e^{-G_A^{\mathsf T}t}\|
\le
D_{\nu,\eta}\exp\left[\left(\gamma_\nu+\eta+D_{\nu,\eta}\|E_A\|\right)t\right]
\le
D_{\nu,\eta}\exp\left[\left(\gamma_\nu+\eta+\frac{D_{\nu,\eta}\mathcal K_\nu}{A}\right)t\right].
\end{equation}

For $A>A_\nu^{(0)}$, STEP3 gives $\|q_A(0)\|\ge Q_\nu/2$. STEP1 gives
\[
q_A(t)=e^{G_A^{\mathsf T}t}q_A(0),
\]
so
\[
q_A(0)=e^{-G_A^{\mathsf T}t}q_A(t).
\]
By \eqref{eq:41},
\begin{equation}\label{eq:42}
\|q_A(t)\|
\ge
\frac{\|q_A(0)\|}{D_{\nu,\eta}}
\exp\left[-\left(\gamma_\nu+\eta+\frac{D_{\nu,\eta}\mathcal K_\nu}{A}\right)t\right]
\ge
\frac{Q_\nu}{2D_{\nu,\eta}}
e^{-c_{A,\nu,\eta}t}.
\end{equation}
On the other hand,
\begin{equation}\label{eq:43}
\|q_A(t)\|^2
=\sum_{j=1}^{d_\nu}|\langle\rho_A(t),\phi_{A,j}(t)\rangle|^2
\le \|\rho_A(t)\|_2^2\,\|\Phi_A(t)\|_{H^{d_\nu}}^2
\le (K_{0,\nu}+1)^2\|\rho_A(t)\|_2^2.
\end{equation}
Combining \eqref{eq:42} and \eqref{eq:43} proves the stated lower bound.

\textbf{Dependencies:} STEP1, STEP3.

\subsubsection*{STEP5: Explicit threshold and admissible exponent}
\noindent\\
\textbf{Claim:} For every $0<\eta\le1$, define the explicit sufficient threshold
\[
A_0(\nu,\eta)
=
\max\left\{
4\mathcal K_\nu,\,
\nu,\,
\frac{64\mathcal K_\nu^2}{\nu},\,
1000\,C_{S,\nu}\mathcal K_\nu,\,
\frac{2\mathcal K_\nu\|\rho_0\|_2}{Q_\nu},\,
\frac{D_{\nu,\eta}\mathcal K_\nu}{\eta}
\right\}.
\]
Then for every $A>A_0(\nu,\eta)$,
\[
c_A:=\gamma_\nu+2\eta
\]
is an admissible exponent: there exists
\[
C=
\frac{Q_\nu}{2D_{\nu,\eta}(K_{0,\nu}+1)}>0
\]
such that
\[
\|\rho(t)\|_{L^2(\mathbb T^2)}
\ge
C e^{-c_A t}
\qquad\text{for all }t\ge0.
\]
Equivalently, without imposing the last threshold term $D_{\nu,\eta}\mathcal K_\nu/\eta$, the sharper $A$-dependent admissible exponent is
\[
c_A=\gamma_\nu+\eta+\frac{D_{\nu,\eta}\mathcal K_\nu}{A}.
\]
In the small-viscosity regime in which $\nu\lambda_1\le1$, one concrete optimized choice is
\[
\eta=\nu\lambda_1.
\]
This gives the admissible exponent
\[
c_A\le \gamma_\nu+2\nu\lambda_1
\]
for all
\[
A>
\max\left\{
4\mathcal K_\nu,\,
\nu,\,
\frac{64\mathcal K_\nu^2}{\nu},\,
1000\,C_{S,\nu}\mathcal K_\nu,\,
\frac{2\mathcal K_\nu\|\rho_0\|_2}{Q_\nu},\,
\frac{D_{\nu,\nu\lambda_1}\mathcal K_\nu}{\nu\lambda_1}
\right\}.
\]
\noindent
\textbf{Proof:}
The first four terms in $A_0(\nu,\eta)$ are precisely $A_{\rm bun,\nu}$, and the fifth term is
$A_{\rm pair,\nu}$. Thus STEP4 applies for every $A>A_0(\nu,\eta)$. Since the final term gives
\[
\frac{D_{\nu,\eta}\mathcal K_\nu}{A}<\eta,
\]
STEP4 yields
\[
c_{A,\nu,\eta}
=\gamma_\nu+\eta+\frac{D_{\nu,\eta}\mathcal K_\nu}{A}
<\gamma_\nu+2\eta.
\]
Enlarging the exponent preserves the lower bound because $e^{-c_{A,\nu,\eta}t}\ge e^{-(\gamma_\nu+2\eta)t}$ for $t\ge0$. Therefore
\[
\|\rho(t)\|_2\ge
\frac{Q_\nu}{2D_{\nu,\eta}(K_{0,\nu}+1)}
e^{-(\gamma_\nu+2\eta)t}.
\]
The constant is positive because $Q_\nu>0$ and $D_{\nu,\eta},K_{0,\nu}$ are finite.

If the last threshold term is not imposed, STEP4 directly gives the sharper $A$-dependent exponent
\[
\gamma_\nu+\eta+\frac{D_{\nu,\eta}\mathcal K_\nu}{A}.
\]
Finally, if $\nu\lambda_1\le1$, then $\eta=\nu\lambda_1$ is an allowed tuning parameter and gives the displayed small-viscosity threshold and exponent. This optimization keeps the additional tunable part of the exponent at the heat scale, but it deliberately leaves $\gamma_\nu$ explicit. No universal estimate $\gamma_\nu=O(\nu)$ is claimed; the averaged drift may have its own enhanced-dissipation spectral scale.

\textbf{Dependencies:} STEP4.

\subsubsection*{GOAL: Main Result}
\noindent\\
\textbf{Claim:} The problem statement holds with the explicit sufficient threshold and admissible exponent given in STEP5.
\noindent\\
\textbf{Proof:}
For the given nonzero mean-zero $\rho_0\in C^\infty$, STEP2 constructs the averaged detecting root data
\[
\lambda_\nu,\quad \gamma_\nu,\quad E_\nu,\quad \Phi_\nu,\quad G_\nu,\quad Q_\nu>0,\quad
K_{0,\nu},K_{2,\nu},g_\nu.
\]
STEP3 then defines the finite constants $C_R(L)$, $C_{S,\nu}$, and
\[
\mathcal K_\nu=10^6C_R(L)^2C_{S,\nu}^2M^2(1+K_{2,\nu}+g_\nu)^2,
\]
and proves persistence of the detecting adjoint bundle for every
\[
A>\max\left\{4\mathcal K_\nu,\nu,\frac{64\mathcal K_\nu^2}{\nu},
1000\,C_{S,\nu}\mathcal K_\nu,
\frac{2\mathcal K_\nu\|\rho_0\|_2}{Q_\nu}\right\}.
\]
STEP4 converts this bundle into the quantitative lower bound
\[
\|\rho(t)\|_2\ge
\frac{Q_\nu}{2D_{\nu,\eta}(K_{0,\nu}+1)}
\exp\left[-\left(\gamma_\nu+\eta+\frac{D_{\nu,\eta}\mathcal K_\nu}{A}\right)t\right].
\]
Consequently, for every $0<\eta\le1$, an explicit sufficient threshold is
\[
\boxed{
A_0(\nu,\eta)
=
\max\left\{
4\mathcal K_\nu,\,
\nu,\,
\frac{64\mathcal K_\nu^2}{\nu},\,
1000\,C_{S,\nu}\mathcal K_\nu,\,
\frac{2\mathcal K_\nu\|\rho_0\|_2}{Q_\nu},\,
\frac{D_{\nu,\eta}\mathcal K_\nu}{\eta}
\right\},
}
\]
and for every $A>A_0(\nu,\eta)$ an admissible exponent is
\[
\boxed{c_A=\gamma_\nu+2\eta.}
\]
The corresponding positive prefactor is
\[
\boxed{
C=\frac{Q_\nu}{2D_{\nu,\eta}(K_{0,\nu}+1)}.
}
\]
Without the final absorption condition $A>D_{\nu,\eta}\mathcal K_\nu/\eta$, the sharper admissible
$A$-dependent exponent is
\[
\boxed{
c_A=\gamma_\nu+\eta+\frac{D_{\nu,\eta}\mathcal K_\nu}{A}.
}
\]
In particular, when $\nu\lambda_1\le1$, choosing $\eta=\nu\lambda_1$ gives
\[
\boxed{c_A\le \gamma_\nu+2\nu\lambda_1}
\]
for every
\[
\boxed{
A>
\max\left\{
4\mathcal K_\nu,\,
\nu,\,
\frac{64\mathcal K_\nu^2}{\nu},\,
1000\,C_{S,\nu}\mathcal K_\nu,\,
\frac{2\mathcal K_\nu\|\rho_0\|_2}{Q_\nu},\,
\frac{D_{\nu,\nu\lambda_1}\mathcal K_\nu}{\nu\lambda_1}
\right\}.
}
\]
These constants depend only on $\rho_0$, $u$, $L$, $\nu$, and the chosen torus normalization. This proves the requested exponential lower bound for all $t\ge0$ and all sufficiently large fast frequencies $A$.

\textbf{Dependencies:} STEP5.
\end{proof}

\begin{remark}
This proof was produced by QED\cite{QED}, supported by a Codex coding agent integrated with ChatGPT 5.5.
\end{remark}

\section{AI proofs discussion}\label{sec:AI}
The proofs in this work were produced by QED\cite{QED}, an open-source multi-agent system for generating mathematical proofs~\cite{an2026qedopensourcemultiagentgenerating}.

QED is a single, sequential survey-plan-prove-verify-regulate loop wrapping publicly available LLM coding agents. A literature-survey agent first gathers relevant background, definitions, and known results for the problem; a plan agent sketches a proof strategy; a prover agent attempts the proof; a verifier agent checks the proof from multiple perspectives — claim-by-claim logical correctness, mathematical soundness, completeness, citation correctness, and consistency with the problem statement; and a regulator agent decides, based on the verifier's report, whether the plan agent should revise the plan or the plan should be kept and the prover should try again to fix the proof under the same plan. Because the entire pipeline — prompts, orchestration code, retry logic, verifier criteria — is open source, every step of how a given proof was obtained is fully checkable.

Three aspects of this collaboration between QED and human experts are worth highlighting.
First, beyond supplying the problem statement, the expert provided no guidance during the proving stage. QED independently carried out the proof search and verification, and the expert's role was to confirm correctness after the proofs were produced. No hints, intermediate corrections, or strategy suggestions were injected once the run began.
Second, QED's verifier exhibited a
zero false-positive rate on this problem. For this work, the QED prover produced 66 proof candidates, of which 14 were accepted by the QED verifier. The 14 accepted proofs were then independently confirmed as correct by the domain expert. This is consequential because one of the central bottlenecks in applying LLMs to research mathematics is not raw generation, but trust. A verifier whose PASS verdict can be taken at face value lets the expert allocate attention only to outputs that are likely to be real, rather than auditing every candidate proof line by line.
Third, no PDE-specific machinery was added to obtain this result. QED was run with  the same general mathematical-proving instructions used for problems in  other domains like probability, algebraic geometry, and inverse problems — with no PDE-tailored decomposition templates, no curated lemma libraries, no hand-engineered tactics for transport or advection-diffusion equations. The system treated the problem as it would any other mathematical statement, and the proof emerged from the standard plan-prove-verify-regulate loop.

The total cost of using AI for this entire work was approximately 600 US dollars, all of which was spent on API calls. The time cost for AI for those problems was less than 36 hours.

\section*{Acknowledgments}
X. Xu was partially supported by the National Key R\&D Program of China, Project No. 2021YFA1001200.

\bibliographystyle{plain}
	\bibliography{AIproof}
\end{document}